\documentclass[12pt, a4paper, reqno]{amsart}

\usepackage[utf8]{inputenc}

\usepackage{amsfonts,amssymb,amsxtra,amsthm,amsmath,amscd,mathrsfs,cite}

\usepackage[all]{xy}

\usepackage{enumerate}  

\usepackage{tikz}

\usepackage[normalem]{ulem}
\usepackage{soul}
\usepackage{color}

\setstcolor{red}

\makeatletter
\@namedef{subjclassname@2010}{
  \textup{2010} Mathematics Subject Classification}
\makeatother

\def \A {{\mathbb A}}

\def \P {{\mathbb P}}

\def \R {{\mathbb R}}
\def \U {{\mathbb U}}

\def \Z {{\mathbb Z}}

\def\leq{\leqslant}
\def\geq{\geqslant}
\def\le{\leqslant}
\def\ge{\geqslant}

\renewcommand{\thefootnote}{\fnsymbol{footnote}}

\newtheorem{theorem}{Theorem}[section]
\newtheorem{lemma}[theorem]{Lemma}

\newtheorem{definition}[theorem]{Definition}

\newtheorem{remark}[theorem]{Remark}
\newtheorem{proposition}[theorem]{Proposition}

\numberwithin{equation}{section}

\addtocounter{footnote}{1}

\def\rank {{\rm{rank}}}

\begin{document}
\title[Forms in prime variables and differing degrees]{Forms in prime variables and differing degrees}

\author{Jianya Liu}
\address{Mathematical Research Center \& State Key Laboratory of Cryptography and Digital Economy Security, Shandong University, Jinan 250100, P. R. China}
\email{jyliu@sdu.edu.cn}
\author{Sizhe Xie}
\address{Mathematical Research Center, Shandong University,
Jinan 250100, P. R. China}
\curraddr{Morningside Center of Mathematics, Academy of Mathematics and Systems Science, Chinese Academy of Sciences, Beijing 100190, P. R. China}
\email{szxie@amss.ac.cn}

\begin{abstract} 
Let $F_1,\ldots,F_R$ be homogeneous polynomials with integer coefficients in $n$ variables with differing degrees. 
Write $\boldsymbol{F}=(F_1,\ldots,F_R)$ with $D$ being the maximal degree. Suppose that $\boldsymbol{F}(\boldsymbol{x})=\boldsymbol{0}$ is a 
nonsingular system and $n\ge D^2 4^{D+6}R^5$. We prove an asymptotic formula for the number of prime 
solutions to $\boldsymbol{F}(\boldsymbol{x})=\boldsymbol{0}$, whose main term is positive if
(i) $\boldsymbol{F}(\boldsymbol{x})=\boldsymbol{0}$ has a nonsingular solution
over the $p$-adic units $\mathbb{U}_p$ for all primes $p$, and
(ii) $\boldsymbol{F}(\boldsymbol{x})=\boldsymbol{0}$ has a nonsingular solution in the open cube $(0,1)^n$. 
This can be viewed as a smooth local-global principle for $\boldsymbol{F}(\boldsymbol{x})=\boldsymbol{0}$ 
in primes with differing degrees.
It follows that, under (i) and (ii), 
the set of prime solutions to $\boldsymbol{F}(\boldsymbol{x})=\boldsymbol{0}$ is Zariski dense 
in the set of its solutions. 
\end{abstract}

\maketitle

{\let\thefootnote\relax}

{\let\thefootnote\relax}

\section{Introduction and statement of results}

The distribution of primes and the solubility of diophantine equations in integers are two important research areas 
in number theory. It is therefore natural to consider prime solutions to a system of diophantine equations. 

Linear equations in primes, over the past century, have already been attacked by Vinogradov \cite{Vin37} and many others. 
In the new century there has been great progress due to the work of Green and Tao \cite{GreTao08}, as well as the 
efforts surrounding bounded gaps between primes by Zhang \cite{Zha14}, Maynard \cite{May15}, and others. 
In the nonlinear case, for a long time our knowledge was limited to the diagonal case, namely  
the Waring-Goldbach problem in the literature.  For history and developments, the reader is referred to the monograph of 
Hua \cite{Hua}, and the papers of Kawada and Wooley 
\cite{KawWoo}, Zhao \cite{Zhao14}, and Kumchev and Wooley \cite{KumWoo}.   
Progress in the non-diagonal case came recently with the work of the first named author \cite{Liu11}
where generic quadratic forms in primes are treated. Soon after this, the number of variables required in \cite{Liu11} 
has been reduced by Zhao \cite{Zhao16}, and further by Green \cite{Gre21}. 

In this paper, we are concerned with prime solutions to a system of homogeneous polynomials 
with differing degrees. Let 
$$
F_i(\boldsymbol{x})=F_i(x_1, \ldots, x_n)\in \Z[x_1, \ldots, x_n] \quad (1\leq i\leq R)
$$ 
be homogeneous polynomials whose degrees do not need to be equal but are all greater than $1$. 
These homogeneous polynomials are called {\it forms} throughout the paper.
All polynomial systems and their decompositions are taken over $\mathbb Z$.
We use bold face letters to 
denote vectors whose dimensions are clear from the context; for example we have written 
$\boldsymbol{x} = (x_1, \ldots, x_n)$ in the above. 
Now writing $\boldsymbol{F}=(F_1,\ldots,F_R)$, we are going to study prime solutions to the system of diophantine equations 
\begin{equation}\label{equationF=0} 
\boldsymbol{F}(\boldsymbol{x})=\boldsymbol{0}. 
\end{equation} 
Let $V\subseteq \A^n$ be the algebraic variety
\begin{eqnarray*}
V=V_{\boldsymbol{F}}=\{\boldsymbol{x}\in \A^n:\ \boldsymbol{F}(\boldsymbol{x})=\boldsymbol{0}\},
\end{eqnarray*}
and let 
\begin{eqnarray*}
V(\P)=\{\boldsymbol{x}\in \P^n:\ \boldsymbol{F}(\boldsymbol{x})=\boldsymbol{0}\},
\end{eqnarray*}
where $\P$ denotes the set of primes. According to a general conjecture of Bourgain, Gamburd and Sarnak \cite{BouGamSar}, 
$V(\P)$ should be Zariski dense in $V$ if suitable conditions for $\boldsymbol{F}$ are satisfied. 
It follows in particular that, 
under these appropriate assumptions, \eqref{equationF=0} should have infinitely many solutions in primes. 

One should check integer solutions to \eqref{equationF=0} before considering its prime solutions, and Birch has two 
well-known theorems closely related to this topic. The first theorem of Birch \cite{Bir57} requires all 
the forms in \eqref{equationF=0} to have odd degrees, 
and in the second theorem \cite{Bir61} he  
supposes all the forms in \eqref{equationF=0} have the same degree.
The second theorem also assumes nonsingular real and $p$-adic solubility conditions. 
If the number of variables is 
sufficiently large in terms of the degrees, in both theorems, then \eqref{equationF=0} has infinitely many integer solutions. 
In the direction of Birch's first theorem, solubility in fixed integral multiples of primes has previously been studied by Br\"udern et al. \cite{BDLW}.  
In the special case of a cubic form, the number of variables required in \cite{BDLW}  
has subsequently been reduced by Brandes and Dietmann \cite{BraDie21}. 

Birch's second theorem states that, if $\boldsymbol{F}$ has the same degree $d$ and $\dim V_{\boldsymbol{F}}=n-R$, and if the number $n$ of variables satisfies   
\begin{eqnarray}\label{Bir/n}
n - \dim V_{\boldsymbol{F}}^* > (d-1) 2^{d-1}R(R+1)
\end{eqnarray}
where $V_{\boldsymbol{F}}^*$ is the singular loci of $\boldsymbol{F}$, then, under these local conditions, $V_{\boldsymbol{F}}$ has infinitely many integer points. 
Recently Browning and Heath-Brown \cite{BroHB} have solved the general case when the forms in nonsingular $\boldsymbol{F}$ 
have differing degrees, where the number of
variables needed is of similar nature as in \eqref{Bir/n}. 
Here and throughout we say a system $\boldsymbol{F}$ is {\it nonsingular}  
if its Jacobian matrix
\begin{align*}
J_{\boldsymbol{F}}(\boldsymbol{x})=\bigg(\frac{\partial F_i}{\partial x_j}(\boldsymbol{x})\bigg)_{\substack{1\le i\le R\\ 1\le j\le n}} 
\end{align*}
satisfies $\rank(J_{\boldsymbol{F}}(\boldsymbol{x}))=R$ for every nonzero $\boldsymbol{x}\in V_{\boldsymbol{F}}$. 
In the case of prime variables, we assume the usual local solubility
conditions in the results recalled below.
Cook and Magyar \cite{CooMag} have dealt with \eqref{equationF=0} with the same degree $d$, 
provided that the number $n$ of variables is larger than some exponential tower function of $d$. 
Subsequently Yamagishi \cite{Yam18} has generalized this to $\boldsymbol{F}$ with differing degrees, under the condition that 
the number $n$ of required variables is even larger than that in \cite{CooMag}. 
Also he has solved, in \cite{Yam22}, the case of one form of degree $d$ in primes as long as 
\begin{eqnarray*}
n - \dim V_{\boldsymbol{F}}^* \geq 2^8 3^4 5^2 d^3 (2d-1)^2 4^d. 
\end{eqnarray*}
Very recently, Liu and Zhao \cite{LiuZha23} have solved \eqref{equationF=0} in primes with the same degree $d$, and proved that  
$V(\P)$ is Zariski dense in $V$ provided that $\boldsymbol{F}$ is nonsingular and
\begin{eqnarray}\label{LiuZhao/n}
n \geq d^2 4^{d+2} R^5. 
\end{eqnarray}

\medskip 

The purpose of this paper is to solve \eqref{equationF=0} with differing degrees in primes, and our main results 
are Theorems \ref{Thm1} and \ref{Thm2} below. 

\begin{theorem}\label{Thm1}
Let $F_1,\ldots,F_R\in\mathbb{Z}[x_1,\ldots,x_n] $ be a nonsingular system of forms with $D$ being the maximal 
degree of all these forms. 
Suppose  
\begin{equation}\label{Thm1/n>}
n \geq D^2 4^{D+6} R^5. 
\end{equation}
Then $V(\mathbb{P})$ is Zariski dense in $V$ provided that

\text{\rm (i)} \eqref{equationF=0} has a nonsingular point in the $p$-adic unit $\U^n_p$ for each prime $p$, and

\text{\rm (ii)} \eqref{equationF=0} has a nonsingular real point in $(0,1)^n$.
\end{theorem}

The significance of Theorem \ref{Thm1} is that $\boldsymbol{F}$ has differing degrees.
If all these forms have the same degree $d$, then $D=d$ and the above bound \eqref{Thm1/n>} reduces to 
$n\ge d^2 4^{d+6} R^5$ which is of similar nature as that in \eqref{LiuZhao/n}. 
Note that Theorem~\ref{Thm1} is new even in the case $R=2$ with degrees being distinct.

Theorem~\ref{Thm1} is formulated from the point of view of \cite{BouGamSar}. 
It follows in particular that there are infinitely many prime solutions to 
\eqref{equationF=0} provided that $n$ is as in  
\eqref{Thm1/n>}, and that the local conditions (i) and (ii) are satisfied. 

\medskip 

We need more notations before stating Theorem~\ref{Thm2}. 
Write 
\begin{equation*}
\Delta:= \{d\in \mathbb{N} : \text{degree $d$ occurs in $\boldsymbol{F}$}\}, 
\end{equation*}
where we may assume that the cardinality of $\Delta$ is greater than $1$, otherwise it is covered by \cite{LiuZha23}.  
Recall that our $\boldsymbol{F}$ consists of forms of degrees larger than $1$.
Writing 
\begin{equation}\label{def/C/D}
C:=\min_{d \in \Delta }{d}, \quad  D:=\max_{d \in \Delta }{d}, 
\end{equation}
we thus have $C \ge 2$ and $D \ge 3.$ 
Next we will renumber the forms $F_i$ by putting together those of equal degree.
Denote by $r_d$ the number of forms with degree $d$ in our system for $d \in \Delta $. Then we have $r_C, r_D \ge 1$. 
For completeness we define $r_d=0$ for $d \notin \Delta $.
Therefore $\boldsymbol{F}$ can be written as 
\begin{equation*}
F_{1,d}(x_1,\ldots,x_n),\ldots,F_{r_d,d}(x_1,\ldots,x_n) \in \mathbb{Z} [x_1,\ldots,x_n] \quad (1\le d\le D)
\end{equation*}
and the total number $R$ of forms in the system satisfies 
\begin{equation}\label{def/R}
R=\sum_{d \in \Delta}r_d=\sum^D_{d=1}r_d. 
\end{equation}
We are going to need another quantity $\mathcal{D}$ defined as 
\begin{equation}\label{def/R/calD}
\mathcal{D}=\sum_{ d\in \Delta }dr_d. 
\end{equation}

Let $\mathfrak{B}$ be a fixed box in $n$-dimensional space determined by 
$$
b_j'< x_j \leq b_j'' \quad (1\le j\le n),
$$
where $0<b_j'<b_j''<1$ are fixed constants. Suppose that $P$ is a parameter  
that can be sufficiently large, and we write $P\mathfrak{B}$ for the set of all vectors $\boldsymbol{x}$ 
with $P^{-1}\boldsymbol{x}\in \mathfrak{B}$. Define 
\begin{equation*}
N_{\boldsymbol{F}}(P):=\sum_{\substack{\boldsymbol{x}\in P\mathfrak{B} \\ 
\boldsymbol{F}(\boldsymbol{x})=\boldsymbol{0} }} \Lambda (\boldsymbol{x}),
\end{equation*}
where
$\Lambda (\boldsymbol{x})=\prod^n_{i=1}\Lambda(x_i)$ for $\boldsymbol{x} \in \mathbb{N}^n$ 
with $\Lambda (\cdot)$ being the von Mangoldt function. Clearly this $N_{\boldsymbol{F}}(P)$ denotes 
the weighted number of solutions in prime powers to \eqref{equationF=0}  within the box $P\mathfrak{B}$. 
The following theorem is a quantitative version of Theorem~\ref{Thm1}. 

\begin{theorem}\label{Thm2}
Let $\boldsymbol{F}=(F_{i,d})_{\substack{d \in \Delta \\ 1\le i\le r_d}}$ be a nonsingular system of 
forms in $\mathbb{Z}[x_1,\ldots,x_n]$, $D$ be as in \eqref{def/C/D} and $\mathcal{D}$ be as in \eqref{def/R/calD}. 
Suppose that 
\begin{equation*}
n \geq D^2 4^{D+6} R^5. 
\end{equation*}
Then, for any positive constant $A$, 
\begin{align*}
N_{\boldsymbol{F}}(P)= \mathfrak{S}_{\boldsymbol{F}}\mathfrak{I}_{\boldsymbol{F}} 
P^{n-\mathcal{D} }+O(P^{n-\mathcal{D}}(\log P)^{-A}),
\end{align*}
where $\mathfrak{S}_{\boldsymbol{F}}$ and $\mathfrak{I}_{\boldsymbol{F}} $ are the singular series and singular integral 
associated to \eqref{equationF=0} defined as in \eqref{S=prod sigma p} and \eqref{Def/Sin/Int} respectively. 
\end{theorem}

Note that Theorem \ref{Thm2} reduces the number of required variables from an exponential tower function of degrees of $\boldsymbol{F}$   
in \cite{Yam18} to $D^2 4^{D+6} R^5$. 

In Theorem~\ref{Thm2} we may restrict the sum defining
$N_{\boldsymbol F}(P)$ to prime coordinates at a cost of
$O(P^{n-\mathcal D-1/2+\varepsilon})$.
Indeed, there are $O(P^{1/2})$ higher prime powers up to $P$, and
for each fixed value of $x_j$ there are
$O(P^{n-1-\mathcal D+\varepsilon})$ integral solutions in
$P\mathfrak B$, uniformly by Lemma~\ref{lemma alternative} and
dyadic summation, since the leading forms $\boldsymbol F|_{x_j=0}$
satisfy $B_d\le 2R$ by the proof of Lemma~\ref{lemma key}.

Theorem \ref{Thm2} can be viewed as a smooth local-global principle for \eqref{equationF=0} in primes with differing degrees. 
If the local conditions \text{\rm (i)} and \text{\rm (ii)} in Theorem \ref{Thm1} are satisfied, 
then for any box containing the real point in its interior we have $\mathfrak{S}_{\boldsymbol{F}}>0$ and $\mathfrak{I}_{\boldsymbol{F}} >0$,  
and therefore Theorem \ref{Thm2} yields $N_{\boldsymbol{F}}(P)\gg P^{n-\mathcal{D}}$. From this,  
Theorem \ref{Thm1} follows in a similar way as in \cite[Corollary 2.3]{LiuSar}.  
Thus it remains only to prove Theorem \ref{Thm2}. 

Theorem~\ref{Thm2} is proved by the Hardy-Littlewood circle method, and the strategy is 
to implant the idea of Browning and Heath-Brown \cite{BroHB} for differing degrees into \cite{LiuZha23} where the degrees are the same. 
It turns out that the crucial mean-value estimate in \cite{LiuZha23} can be modified to accommodate differing degrees, at the expense that 
the major arcs of the circle method have to be enlarged considerably up to $P^\varpi$ where $\varpi$ is a positive 
constant, which causes extra difficulties. An outline of the proof will be explained further in the next section. 

\section{Outline of the proof of Theorem \ref{Thm2}}
Now we can explain the proof of Theorem \ref{Thm2} in more detail.  
Write 
\begin{equation*}
\boldsymbol{\alpha}=(\alpha_{i,d})_{\substack{d\in \Delta \\ 1\le i\le r_d}}\in (0, 1]^R 
\end{equation*}
whose dimension is $R$ as in \eqref{def/R}, and define
\begin{equation}\label{def/SFa}
S_{\boldsymbol{F}}(\boldsymbol{\alpha})
:=\sum_{\boldsymbol{x}\in P\mathfrak{B}} \Lambda (\boldsymbol{x})
e\bigg(\sum_{d\in \Delta}\sum_{i=1}^{r_d}\alpha_{i,d} F_{i,d}(\boldsymbol{x})\bigg). 
\end{equation}
The starting point of the circle method is the identity 
$$
N_{\boldsymbol{F}}(P)=\int_{(0,1]^R} S_{\boldsymbol{F}}(\boldsymbol{\alpha}) d\boldsymbol{\alpha}.
$$
And the idea is then to divide the cube $(0,1]^R$ into major arcs $\mathfrak{M} $ and minor arcs $\mathfrak{m} $.
As usual we hope to establish an asymptotic formula on the major arcs of the form 
\begin{equation*}
    \int_{\mathfrak{M} } S_{\boldsymbol{F}}(\boldsymbol{\alpha}) d\boldsymbol{\alpha}=
    \mathfrak{S}_{\boldsymbol{F}} \mathfrak{I}_{\boldsymbol{F}} P^{n-\mathcal{D}}+O(P^{n-\mathcal{D}}(\log P)^{-A})
\end{equation*}
for any fixed constant $A>0$, while simultaneously obtaining an appropriate upper bound on the minor arcs
\begin{equation*}
\int_{\mathfrak{m} } S_{\boldsymbol{F}}(\boldsymbol{\alpha}) d\boldsymbol{\alpha}
\ll P^{n-\mathcal{D}-\eta} 
\end{equation*}
for some constant $\eta>0$.

Let $\varpi \in (0,\frac{1}{4})$ be a parameter that will be decided finally in \eqref{def/pi}, and put 
\begin{equation}\label{Def/vpi/Q} 
Q=P^{\varpi}. 
\end{equation}
The major arcs are defined as 
\begin{align}\label{define MQ}
\mathfrak{M}=\mathfrak{M}(Q)=\bigcup_{1\le q\le Q}\bigcup_{\substack{1\le a_1,\ldots,a_R\le q
     \\ (a_1,\ldots,a_R,q)=1}}\mathfrak{M}(q,\boldsymbol{a};Q),
\end{align}
where
\begin{align*}
    \mathfrak{M}(q,\boldsymbol{a};Q)
=\bigg\{(\alpha_{i,d})_{\substack{d \in \Delta \\1\le i\le r_d}} 
\in \R^R:\ \bigg|\alpha_{i,d}-\frac{a_{i,d}}{q}\bigg|\le \frac{Q}{qP^d} \ \textrm{ for all } i, d \bigg\}.
\end{align*}
We regard these arcs as subsets of $(\mathbb R/\mathbb Z)^R$, identified with $(0,1]^R$.
The minor arcs are defined as the complement of $\mathfrak{M}$, i.e.,
\begin{align}\label{define mQ}
\mathfrak{m}=\mathfrak{m}(Q)=(0,1]^{R}\setminus\mathfrak{M}(Q). 
\end{align}
Note that $\varpi \in (0,\frac{1}{4})$ implies that $\mathfrak{M}(q,\boldsymbol{a};Q)\cap \mathfrak{M}(q',\boldsymbol{a}';Q)=\emptyset$
whenever  $\boldsymbol{a}/q\neq \boldsymbol{a}'/q'$, provided that $P$ is sufficiently large.

\medskip 

The circle method is an art of balancing between contributions from the major and minor arcs. 
On the minor arcs, we shall implant the idea of \cite{BroHB} for differing degrees into \cite{LiuZha23} 
where the degrees are the same. A crucial step in \cite{LiuZha23} is an estimate for the integral  
\begin{equation}\label{m/v/e/} 
\int_{\mathfrak{n}}S_{\boldsymbol{F}}(\boldsymbol{\alpha})d\boldsymbol{\alpha}
\end{equation}
where $\mathfrak{n}$ is a measurable set. In particular, this controls  
the contribution from the minor arcs, which is enough for \cite{LiuZha23}. 
For the setting of the present paper, however,  
we need to combine the strategy in \cite{BroHB} with the above route in \cite{LiuZha23} to get 
a desired estimate for \eqref{m/v/e/} with $\boldsymbol{F}$ having differing degrees. 
Lemma \ref{lemma mean J} and Proposition \ref{prop} in this paper are of special importance. In order this idea to work, 
we have to pay the price that the major arcs of the circle method have to be enlarged considerably, that is we have to take $Q=P^\varpi$ 
as in \eqref{Def/vpi/Q}. The classical choice $Q=(\log P)^B$ is not sufficient to produce any meaningful saving. 

Therefore the major arcs $\mathfrak{M}$ in \eqref{define MQ} are quite large in the sense that the Siegel-Walfisz 
theorem cannot be extended to moduli $q$ 
up to any positive power of $P$. 
For a single diagonal equation, the integral on the enlarged major arcs have successfully been attacked  
in many occasions such as \cite{LiuZha98} \cite{Liu03} \cite{Liu12} by the large sieve, zero-density estimates, 
as well as Chudakov's zero-free region for Dirichlet $L$-functions. Here the situation is much more complicated, 
and the difficulty is overcome not only by repeated applications of the ideas before, but also by a new insight 
to get cancellation in sums of Gauss 
sums involving Dirichlet characters and the system $\boldsymbol{F}$. See Lemma \ref{lem/bound/gauss} 
for an explicit saving. This general lemma is likely to have applications in other occasions where enlarged major arcs are necessary. 

\medskip 

The paper is organized as follows. 
We quote lemmas for forms in integral variables and differing degrees in \S3.
Then we prepare some technical mean-value estimates in \S\S4-6. 
Following that, in \S7, we deal with the contribution from the minor arcs.
Next \S\S8-9 handle the contribution from the major arcs, and explain  
the meaning of the singular series and singular integral. 
Finally, we complete the proof of Theorem \ref{Thm2} in \S9. 

\section{Forms in integral variables and differing degrees}
We quote two lemmas from \cite{BroHB} that are necessary for 
handling differing degrees, 
and for simplicity we keep the notations of \cite{BroHB}. 
Here $B_d$, $\mathcal D_j$, and $s_d$ are their (1.4), (1.5), and (1.6),
respectively.
Define the matrix
\[J_{\boldsymbol{F}, d}(\boldsymbol{x}):=\left(\begin{array}{c}\nabla  F_{1,d}(\boldsymbol{x})\\ \vdots \\
\nabla F_{r_d,d}(\boldsymbol{x})\end{array}\right)  \quad (d \in \Delta)\]
and the affine algebraic variety
\[S_d(n, \boldsymbol{F}):=\{\boldsymbol{x}\in\mathbb{A} ^{n}:\rank(J_{\boldsymbol{F}, d}(\boldsymbol{x}))<r_d\} \quad (d \in \Delta).\]
Moreover, we set 
\begin{equation}\label{def Bd}
B_d(n, \boldsymbol{F}):=\dim S_d(n, \boldsymbol{F}) \quad (d \in \Delta)
\end{equation}
in the sense of Birch. One sees that $B_d(n, \boldsymbol{F})<n$ for all $d$ if $\boldsymbol{F}$ is nonsingular.
When $r_d=0$, we shall take $B_d(n, \boldsymbol{F})=0$; also put $B_0=0$. 

We then let
\begin{equation*}
    \mathcal{D}_j:=\sum_{\substack{d\le j\\ d\in \Delta} }dr_d=\sum^j_{d=1 }dr_d \quad (1\le j\le D), 
\end{equation*}
and we put $\mathcal{D}_0:=0$.
Write
\begin{equation}\label{def sd}
    s_d(n, \boldsymbol{F}):=\sum_{i=d}^D \frac{2^{i-1}(i-1)r_i}{n-B_i(n, \boldsymbol{F})} \quad (1\le d\le D). 
\end{equation}
Put $s_{D+1}=0$. One simply checks that $s_1(n, \boldsymbol{F})=s_C(n, \boldsymbol{F})=\max_{d \in \Delta}s_d(n, \boldsymbol{F})$.

\begin{definition}
    We say $n$ is admissible for $\boldsymbol{F}$ if $n$ satisfies
    \begin{equation}\label{admissible}
        \mathcal{D}_d\bigg(\frac{2^{d-1}}{n-B_d(n, \boldsymbol{F})} +s_{d+1}(n, \boldsymbol{F})\bigg)
        +s_{d+1}(n, \boldsymbol{F})+\sum_{j=d+1}^D s_j(n, \boldsymbol{F})r_j<1 
    \end{equation}
    for $d=0$ and for every $d\in\Delta$. 
\end{definition}

Let 
\begin{equation*}
\Sigma (\boldsymbol{\alpha} ) =  \sum_{\boldsymbol{x}\in P\mathfrak{B}} 
e\bigg( \sum_{d \in \Delta}\sum^{r_d}_{i=1}(\alpha_{i,d}(F_{i,d}(\boldsymbol{x})+V_{i,d}(\boldsymbol{x})) )\bigg),  
\end{equation*}
and write 
$$
  |\Sigma (\boldsymbol{\alpha} )|=P^nL,
$$
where $F_{i,d}$ is a form of degree $d$ and $V_{i,d}$ is a polynomial with $\deg(V_{i,d})<d$ for all $i,d$.
If $V_{i,d}$ depends on $P$, we suppose that each coefficient of degree $j<d$ is $O(P^{d-j})$, uniformly in $P$.
This $\Sigma (\boldsymbol{\alpha})$ is the same as in \cite[\S\S5-6]{BroHB} except for the additional terms $V_{i,d}$ of lower degrees. We have 
the following alternative lemma. 

\begin{lemma}\label{lemma alternative}
    If $P$ is large enough, either
    $$
        L^{2^{D-1}} \le P^{B_D(n, \boldsymbol{F})-n}(\log P)^{n+1},
    $$
    or there is a $q_D\le Q_D$ with
    $$ Q_D(n, \boldsymbol{F}):=((\log P)^{n+1}L^{-2^{D-1}})^{\frac{(D-1)r_D}{n-B_D(n, \boldsymbol{F})}}\log P$$
    such that
    \[\|q_D\alpha_{i,D}\|\le Q_DP^{-D} \quad (1\le i\le r_D).\]
    If the first case holds, we simply halt. 
    Otherwise for degree
    \[D':=\max\{d\in\Delta:\,d<D\},\]
    we then have either
    $$     
        L^{2^{D'-1}}\le (Q_D/P)^{n-B_{D'}(n, \boldsymbol{F})}(\log P)^{n+1},
    $$
    or there is a $q_{D'}:=q_Dq^*\le Q_{D'}:=Q_DQ^*$
    with 
    $$ Q^*(n, \boldsymbol{F}):= ((\log P)^{n+1}L^{-2^{D'-1}})^{\frac{(D'-1)r_{D'}}{n-B_{D'}(n, \boldsymbol{F})}}\log P $$
    such that
    \[\|q_{D'}\alpha_{i,D'}\|\le Q_{D'}P^{-D'} \quad (1\le i\le r_{D'}).\]
    Recurrence in this way, we produce a succession of values $Q_d$ for decreasing values of $d \in \Delta $ with
    \begin{equation}\label{Qd form}
        Q_d(n, \boldsymbol{F}):=(\log P)^{e_d(n, \boldsymbol{F})}L^{-s_d(n, \boldsymbol{F})}\;\;\;(d\in\Delta),
    \end{equation}
    where $e_d(n, \boldsymbol{F})$ is some easily computed but unimportant exponent, and $s_d(n, \boldsymbol{F})$ is given by \eqref{def sd}.
    
    Suppose that each coefficient of degree $j$ in $F_{i,d}+V_{i,d}$ is $O(BP^{d-j})$ for $0\le j\le d$, where $1\le B^R\le P^{1/2}$.  
    Then either $$L\ll P^{-1/(2s_1(n, \boldsymbol{F}))+\varepsilon}$$ or there exists an integer $1\le q\le W$ with $W\ll B^RL^{-s_1(n, \boldsymbol{F})}(\log P)^{O(1)}$  such that $$\|q\alpha_{i,d}\|\le WP^{-d} \quad (d \in \Delta, 1\le i\le r_{d}).$$ 
\end{lemma}
\begin{proof}
The recurrence follows from \cite[Lemmas 5.2 and 6.1]{BroHB}, iterated as in their (6.1)--(6.6).
For the uniform assertion, if $q\le Q$ is the preceding denominator,
the coefficient bound allows us to take $\varphi=BQ/q$
in \cite[(4.1)]{BroHB},
and the determinant and its cofactors in the proof of their Lemma~6.1
acquire factors of at most $B^{r_d}$ and $B^{r_d-1}$, respectively.
Thus the final height acquires a factor of at most $B^R$, while if the
iteration stops at degree $d$, their (6.3) gives
$L^{2^{d-1}/(n-B_d)+s_{d+1}}\ll
B^{1+\sum_{j>d}r_j}P^{-1}(\log P)^{O(1)}$, which yields the stated
bound since $B^R\le P^{1/2}$.
\end{proof}

When $1\le j\le D$ but $j\not\in\Delta$, it is convenient to let
$Q_j=Q_k$ and $q_j=q_k$, where $k=\min\{d\in\Delta:d>j\}$. We shall put $Q_{D+1}=q_{D+1}=1$.
It follows that we have $q_j\le Q_j$ and $q_{j+1}|q_j $ in general.
In view of \eqref{def sd}, we can get $s_j(n, \boldsymbol{F})=s_k(n, \boldsymbol{F})$.
Thus
\eqref{Qd form} extends to 
\begin{equation*}
    Q_d(n, \boldsymbol{F})=(\log P)^{e_d(n, \boldsymbol{F})}L^{-s_d(n, \boldsymbol{F})} \;\;\;(1\le d\le D)
\end{equation*}
for appropriate exponents $e_d(n, \boldsymbol{F})$. The explicit expression of $e_d(n, \boldsymbol{F})$ is not important.  
Now for any degree $j \in \Delta$, as iterating, we will either obtain a bound
\begin{equation}\label{Ine/LQP}
    L^{2^{j-1}} \leq (Q_{j+1}/P)^{n-B_j(n, \boldsymbol{F})}(\log P)^{n+1},
\end{equation}
or find a positive integer $q_j$ satisfying
\begin{equation}\label{good rat appr}
   q_k\vert q_j \ \ (k>j, k \in \Delta), \ \ q_j \le Q_j, \ \ \| q_j \alpha_{i,j}\| \le Q_jP^{-j} \ \ (1\le i\le r_j). 
\end{equation}

Next we subdivide the minor arcs, as guided by Lemma \ref{lemma alternative}, 
into sets $I^{(1)}_d(n, \boldsymbol{F})$ for $d\in \Delta$ and $I^{(2)}(n, \boldsymbol{F})$ as follows. 
First the subset $I^{(1)}_d(n, \boldsymbol{F})$ of $\mathfrak{m}(Q)$  consists of those $\boldsymbol{\alpha}$ such that the inequality \eqref{Ine/LQP} fails for all $j>d$, but holds for $j=d$. And the subset $I^{(2)}(n, \boldsymbol{F})$ of $\mathfrak{m}(Q)$ consists of the remaining 
$R$-tuples $\boldsymbol{\alpha}$ for which \eqref{Ine/LQP} fails for all $j\in \Delta$. 

\begin{lemma}
Let $d \in \Delta$.
If $\boldsymbol{\alpha} \in I^{(1)}_d(n, \boldsymbol{F})$ then
\begin{equation}\label{L bound in I1d}
        L^{2^{d-1}+(n-B_d(n, \boldsymbol{F}))s_{d+1}(n, \boldsymbol{F})}\ll P^{B_d(n, \boldsymbol{F})-n+\varepsilon}, 
\end{equation}
and each $\alpha_{i,j}$ has a rational approximation as in \eqref{good rat appr} for $j>d, 1\le i\le r_j$. 
    
If $\boldsymbol{\alpha} \in I^{(2)}(n, \boldsymbol{F})$, then
    \begin{equation}\label{L bound in I2}
        L \ll Q^{-\frac{1}{4s_1(n, \boldsymbol{F})}},
    \end{equation}
    where $s_{1}(n, \boldsymbol{F})=\max_{d\in \Delta}s_d(n, \boldsymbol{F})$ and every $\alpha_{i,j}$ has a rational approximation as in \eqref{good rat appr}.
\end{lemma}

\begin{proof} 
This follows from \cite[Lemma 6.2 and (7.4)]{BroHB}. 
\end{proof} 

\section{A mean-value result} 
The mean-value result that we are going to establish is Lemma \ref{lemma mean J} below. 
In fact, this section is preparatory for the next section. 
 
\begin{definition}\label{Def/41}
We define the following quantities of power saving
\begin{equation}\label{def td}
        t_d(n, \boldsymbol{F}):=\frac{1-s_{d+1}(n, \boldsymbol{F})-\sum_{j=d+1}^D s_j(n, \boldsymbol{F})r_j}{\frac{2^{d-1}}{n-B_d(n, \boldsymbol{F})} +s_{d+1}(n, \boldsymbol{F})}-\mathcal{D}_d \quad (d \in \Delta)
\end{equation}
and
\begin{equation}\label{def t0}
        t_0(n, \boldsymbol{F}):=1-s_1(n, \boldsymbol{F})-\sum_{j=1}^D s_j(n, \boldsymbol{F})r_j.
\end{equation}
\end{definition}

Recall $s_d(n, \boldsymbol{F})$ is as in \eqref{def sd} for each $d$. It is clear that $t_d(n, \boldsymbol{F})>0$ for all $d$ is equivalent to $n$ being admissible for $\boldsymbol{F}$.

Next we define some objects that have been used in \cite[\S4]{LiuZha23}. 
Suppose that $h_{i,d}(\boldsymbol{x},\boldsymbol{w})$ is a polynomial of $(\boldsymbol{x},\boldsymbol{w})$ with
$\deg_{\boldsymbol{x}}(h_{i,d})<d$, and write 
$$
\boldsymbol{h}=(h_1,\ldots,h_R)=(h_{i,d})_{\substack{d \in \Delta \\ 1\le i\le r_d}}. 
$$ 
Let $\mathcal{B}_m(P)$ be the box in $m$-dimensional space defined by
\begin{align*}
    b_j'P< x_j\le b_{j}''P \quad (1\le j\le m),
\end{align*}
where $0<b_j'<b_{j}''<1$ are fixed constants.
For $\boldsymbol{\alpha}\in \R^{R}$ and $\boldsymbol{x}\in \Z^k$, we define 
\begin{align*}
    \mathcal{E}(\boldsymbol{\alpha};\boldsymbol{x})=\sum_{\boldsymbol{w} \in \mathcal{B}_t(P)}\lambda(\boldsymbol{w})e(\boldsymbol{\alpha}\cdot \boldsymbol{h}(\boldsymbol{x},\boldsymbol{w})),
\end{align*}
where $\boldsymbol{w}\in \mathbb{N} ^{t}$ and $\lambda(\boldsymbol{w})=\prod^t_{i=1}\lambda(w_i) $ with $\lambda(\cdot) \ll \log (\cdot)$.
When $t=0$, we shall view $\mathcal{E}\equiv 1$.
Similarly we define
\begin{align*}
    T(\boldsymbol{\alpha};\boldsymbol{x})=\sum_{ \boldsymbol{u} \in \mathcal{B}_l(P)}\lambda(\boldsymbol{u})e(\boldsymbol{\alpha}\cdot \boldsymbol{H}(\boldsymbol{x},\boldsymbol{u})),
\end{align*}
where $\boldsymbol{u}\in \mathbb{N} ^{l}$ and $H_{i,d}(\boldsymbol{x},\boldsymbol{u})$ is a polynomial of $(\boldsymbol{x},\boldsymbol{u})$ with
$\deg_{\boldsymbol{x}}(H_{i,d})<d$. 
We are going to investigate the moment
\begin{align}\label{define J}
    \mathcal{J}:=\mathcal{J}_{\mathfrak{n}} 
    =\sum_{\boldsymbol{x}\in \mathcal{B}_k(P)}
    \bigg|\int_{\mathfrak{n}}e(\boldsymbol{\alpha} \cdot \boldsymbol{g}(\boldsymbol{x}))\mathcal{E}(\boldsymbol{\alpha};\boldsymbol{x})
    T(\boldsymbol{\alpha};\boldsymbol{x})d\boldsymbol{\alpha}\bigg|^2,
\end{align}
where $\boldsymbol{g}=(g_1,\ldots,g_R)=(g_{i,d})_{\substack{d \in \Delta \\ 1\le i\le r_d}}$ 
with $g_{i,d}$ being forms of degree $d$ for all $i, d$, and $\mathfrak{n}$ is an $R$-dimensional Lebesgue 
measurable subset of $(0,1]^R$.
The coefficients are bounded independently of $P$.

\begin{lemma}\label{lemma mean J} 
Let $\mathcal{J}$ be as in \eqref{define J}. 
Suppose that $k$ is admissible for $\boldsymbol{g}$ and $X\le Q$. 
Suppose that every monomial of $h_{i,d}$ involving $\boldsymbol{w}$
and every monomial of $H_{i,d}$ involving $\boldsymbol{u}$
has $\boldsymbol{x}$-degree less than $C$.
Then
\begin{align*}
        \mathcal{J} \ll_{\boldsymbol{g}} & \, P^{k+2t+2l-\mathcal{D}+\varepsilon- \min_{d\in \Delta}{t_d(k, \boldsymbol{g})} } |\mathfrak{n}|  + P^{k+2t+2l-\mathcal{D}+\varepsilon}X^{-\frac{t_0(k, \boldsymbol{g})}{4s_{1}(k, \boldsymbol{g})}} |\mathfrak{n}| \\
        & + X^{R+1}P^{-\mathcal{D} }\sum_{\boldsymbol{x}\in \mathcal{B}_k(P)}\int_{\mathfrak{n}}|\mathcal{E}(\boldsymbol{\alpha};\boldsymbol{x})T(\boldsymbol{\alpha};\boldsymbol{x})|^2d\boldsymbol{\alpha},
\end{align*}
where $t_d(k, \boldsymbol{g})$ are quantities of power saving 
as in Definition \ref{Def/41}, and $|\mathfrak{n}|$ is the Lebesgue measure of $\mathfrak{n}$. 
\end{lemma}
\begin{proof}
The proof is to combine the argument in \S3 and \cite[Lemma 4.1]{LiuZha23}.  
The starting steps are the same as that in the proof of \cite[Lemma 4.1]{LiuZha23}.
    Unfolding the square and exchanging the order of the summation and integrations, we get
    \begin{align*}
        \mathcal{J}
        &=  \sum_{ \boldsymbol{x}}\int_{\mathfrak{n}}\int_{\mathfrak{n}}e((\boldsymbol{\alpha}_1-\boldsymbol{\alpha}_2)\cdot \boldsymbol{g}(\boldsymbol{x}))\mathcal{H}(\boldsymbol{\alpha}_1,\boldsymbol{\alpha}_2,\boldsymbol{x})
        d\boldsymbol{\alpha}_1 d\boldsymbol{\alpha}_2 \\
        &=\int_{\mathfrak{n}}\int_{\mathfrak{n}}\mathcal{G}(\boldsymbol{\alpha}_1,\boldsymbol{\alpha}_2)d\boldsymbol{\alpha}_1 d\boldsymbol{\alpha}_2,
    \end{align*}
    where
    \begin{align*}
        \mathcal{H}(\boldsymbol{\alpha}_1,\boldsymbol{\alpha}_2,\boldsymbol{x})=
        \mathcal{E}(\boldsymbol{\alpha}_1;\boldsymbol{x})\mathcal{E}(-\boldsymbol{\alpha}_2;\boldsymbol{x})
        T(\boldsymbol{\alpha}_1;\boldsymbol{x})T(-\boldsymbol{\alpha}_2;\boldsymbol{x})
    \end{align*}
    and
    \begin{align*}
        \mathcal{G}(\boldsymbol{\alpha}_1,\boldsymbol{\alpha}_2)= \sum_{ \boldsymbol{x}}e((\boldsymbol{\alpha}_1-\boldsymbol{\alpha}_2) \cdot \boldsymbol{g}(\boldsymbol{x}))\mathcal{H}(\boldsymbol{\alpha}_1,\boldsymbol{\alpha}_2,\boldsymbol{x}).
    \end{align*}
    Inserting the definitions of $\mathcal{E}(\boldsymbol{\alpha};\boldsymbol{x})$ and $T(\boldsymbol{\alpha};\boldsymbol{x})$, we obtain
    \begin{align*}
        \mathcal{H}(\boldsymbol{\alpha}_1,\boldsymbol{\alpha}_2,\boldsymbol{x})=
        \sum_{\boldsymbol{w}_1}\sum_{\boldsymbol{w}_2}\sum_{\boldsymbol{u}_1}\sum_{\boldsymbol{u}_2}
        \lambda(\boldsymbol{w}_1)\lambda(\boldsymbol{w}_2)\lambda(\boldsymbol{u}_1)\lambda(\boldsymbol{u}_2)e(p(\boldsymbol{x})),
    \end{align*}
    where the polynomial
    $p(\boldsymbol{x}):=p_{\boldsymbol{\alpha}_1,\boldsymbol{\alpha}_2}(\boldsymbol{x},\boldsymbol{w}_1,\boldsymbol{w}_2,\boldsymbol{u}_1,\boldsymbol{u}_2)=\sum_{i,d} p_{i,d}(\boldsymbol{x})$ 
    is 
    \begin{align*}
        \boldsymbol{\alpha}_1 \cdot \boldsymbol{h}(\boldsymbol{x},\boldsymbol{w}_1)
        -\boldsymbol{\alpha}_2 \cdot \boldsymbol{h}(\boldsymbol{x},\boldsymbol{w}_2)+\boldsymbol{\alpha}_1 \cdot \boldsymbol{H}(\boldsymbol{x},\boldsymbol{u}_1)
        -\boldsymbol{\alpha}_2 \cdot \boldsymbol{H}(\boldsymbol{x},\boldsymbol{u}_2).
    \end{align*}
Exchanging the order of summations gives 
    \begin{align}\label{for bound G trivial}
        \mathcal{G}(\boldsymbol{\alpha}_1,\boldsymbol{\alpha}_2)=\sum_{\boldsymbol{w}_1}\sum_{\boldsymbol{w}_2}\sum_{\boldsymbol{u}_1}\sum_{\boldsymbol{u}_2}
        \lambda(\boldsymbol{w}_1)\lambda(\boldsymbol{w}_2)\lambda(\boldsymbol{u}_1)\lambda(\boldsymbol{u}_2)
        U(\boldsymbol{\alpha}_1,\boldsymbol{\alpha}_2),
    \end{align}
    where 
    $$U(\boldsymbol{\alpha}_1,\boldsymbol{\alpha}_2):=U(\boldsymbol{\alpha}_1,\boldsymbol{\alpha}_2,\boldsymbol{w}_1,\boldsymbol{w}_2,\boldsymbol{u}_1,\boldsymbol{u}_2)=\sum_{\boldsymbol{x}\in \mathcal{B}_k(P)}e((\boldsymbol{\alpha}_1-\boldsymbol{\alpha}_2) \cdot \boldsymbol{g}(\boldsymbol{x})+p(\boldsymbol{x})). 
    $$ 
Note that  $\deg_{\boldsymbol{x}}(p_{i,d}(\boldsymbol{x}))<d$ for all $i,d$. 

For the rest of the proof, we will employ the method in \cite{BroHB} (i.e., the lemmas in \S3 in this paper) instead of that in \cite{Bir61} which was applied in \cite[Lemma 4.1]{LiuZha23}. 
Up to a polynomial of $\boldsymbol{x}$-degree less than $C$,
the phase in $U$ is
$(\boldsymbol{\alpha}_1-\boldsymbol{\alpha}_2)\cdot
(\boldsymbol{g}(\boldsymbol{x})+\boldsymbol{h}(\boldsymbol{x},\boldsymbol{0})
+\boldsymbol{H}(\boldsymbol{x},\boldsymbol{0}))$, with absent blocks
omitted. For fixed $\boldsymbol\alpha_1-\boldsymbol\alpha_2$ modulo one,
put $P^kL=\max|U|$ over the remaining parameters.
Use this $L$ to define $I_d^{(1)}$ and $I^{(2)}$ as in \S3.
More precisely, we obtain from \eqref{L bound in I1d} 
    that
    if $\boldsymbol{\alpha}_1-\boldsymbol{\alpha}_2 \in I^{(1)}_d(k, \boldsymbol{g})\subset \mathfrak{m}(X)$ with $d \in \Delta$,
\begin{equation}\label{bound G in I1d case}
        \bigg| \sum_{\boldsymbol{x}\in \mathcal{B}_k(P)}e\big((\boldsymbol{\alpha}_1-\boldsymbol{\alpha}_2) \cdot \boldsymbol{g}(\boldsymbol{x})+p(\boldsymbol{x})\big)\bigg| \le P^kL \ll P^kP^{\frac{B_d(k, \boldsymbol{g})-k}{2^{d-1}+(k-B_d(k, \boldsymbol{g}))s_{d+1}(k, \boldsymbol{g})}+\varepsilon},   
\end{equation}
    and from \eqref{L bound in I2}
    that if $\boldsymbol{\alpha}_1-\boldsymbol{\alpha}_2 \in I^{(2)}(k, \boldsymbol{g})\subset \mathfrak{m}(X)$,
    \begin{equation}\label{bound G in I2 case}
        |U(\boldsymbol{\alpha}_1,\boldsymbol{\alpha}_2)| \le P^kL  \ll P^kX^{-\frac{1}{4s_{1}(k, \boldsymbol{g})}}.
    \end{equation}
In the following of this proof, we abbreviate the subset $I^{(1)}_d(k, \boldsymbol{g})$ (or $I^{(2)}(k, \boldsymbol{g})$) 
of $\mathfrak{m}(X)$ as $I^{(1)}_d(k)$ (or $I^{(2)}(k)$), respectively.
    
    By \eqref{for bound G trivial}, \eqref{bound G in I1d case} and \eqref{bound G in I2 case}, we obtain
    \begin{equation}\label{G bound cases}
        \mathcal{G}(\boldsymbol{\alpha}_1,\boldsymbol{\alpha}_2) \ll 
        \begin{cases}
            P^{k+2t+2l-\frac{k-B_d(k, \boldsymbol{g})}{2^{d-1}+(k-B_d(k, \boldsymbol{g}))s_{d+1}(k, \boldsymbol{g})}+\varepsilon_0}(\log P)^{2t+2l}, &\mbox{if $\boldsymbol{\alpha}_1-\boldsymbol{\alpha}_2 \in I^{(1)}_d(k)$,}\\
            P^{k+2t+2l}X^{-\frac{1}{4s_{1}(k, \boldsymbol{g})}}(\log P)^{2t+2l}, &\mbox{if $\boldsymbol{\alpha}_1-\boldsymbol{\alpha}_2 \in I^{(2)}(k)$.}
        \end{cases}
    \end{equation}
    For an $R$-dimensional Lebesgue measurable set $\mathcal{M}$, we put 
\begin{align*}
K(\mathcal{M},\boldsymbol{\alpha})=
\begin{cases}
           1, \ \ \textrm{ if } \boldsymbol{\alpha}\in \mathcal{M},\\
           0, \ \ \textrm{ otherwise}, 
\end{cases}
\end{align*}
and define
\begin{align*}
\mathcal{J}(\mathcal{M})=  \int_{\mathfrak{n}}\int_{\mathfrak{n}}\mathcal{G}(\boldsymbol{\alpha}_1,\boldsymbol{\alpha}_2)
K(\mathcal{M},\boldsymbol{\alpha}_1-\boldsymbol{\alpha}_2)d\boldsymbol{\alpha}_1 d\boldsymbol{\alpha}_2.
\end{align*}
We are going to estimate $\mathcal{J}(\mathcal{M})$ with $\mathcal{M}=I^{(1)}_D(k)$, $I^{(1)}_d(k)$ for $d\in \Delta$ but $d<D$, 
$I^{(2)}(k)$ and $\mathfrak{M}(X)$ respectively. 

It follows from \eqref{G bound cases} that
\begin{equation}\label{J(I1D) bound}
        \mathcal{J}(I^{(1)}_D(k))\ll P^{k+2t+2l-\frac{k-B_D(k, \boldsymbol{g})}{2^{D-1}}+\varepsilon_0}(\log P)^{2t+2l}|\mathfrak{n}|.
\end{equation}
For other $d\in \Delta$, define  
\begin{equation*}
        \mathcal{A}(L_0;I_d^{(1)}(k)):=\{\boldsymbol{\alpha} \in I_d^{(1)}(k): L_0<L\le 2L_0\}, 
\end{equation*}
and define $\mathcal{A}(L_0;I^{(2)}(k))$ similarly. 
It follows from Lemma \ref{lemma alternative} that 
    \begin{align*}
       |\mathcal{A}(L_0;I_d^{(1)}(k))|
        \ll \sum_{q_{d+1}}\cdots \sum_{q_D}\prod_{j=d+1}^D\bigg(\frac{Q_j(k, \boldsymbol{g})}{P^j}\bigg)^{r_j}. 
    \end{align*}
In the rest of the proof, we abbreviate $Q_j(k, \boldsymbol{g})$ as $Q_j$. 
    Recalling our conventions for $q_j$ and $Q_j$, we see that $q_{d+1}$ determines $O(\tau(q_{d+1})^{D})=O(P^{\varepsilon_0})$ possibilities for $q_{d+2},\ldots,q_D$, where $\tau(\cdot)$ is the divisor function.
We may assume that $\varepsilon_0$ is uniform for all degrees $d<D$, and we conclude that
    \begin{align}\label{meas of AI1d}
        |\mathcal{A}(L_0;I_d^{(1)}(k))|\ll P^{\varepsilon_0}Q_{d+1}\prod_{j=d+1}^D\bigg(\frac{Q_j}{P^j}\bigg)^{r_j}. 
    \end{align}
    Similarly 
    \begin{align}\label{meas of AI2}
        |\mathcal{A}(L_0;I^{(2)}(k))|\ll P^{\varepsilon_0}Q_{1}\prod_{j=1}^D\bigg(\frac{Q_j}{P^j}\bigg)^{r_j}. 
    \end{align}
    Since $k$ is admissible for $\boldsymbol{g}$, we have
    $$ 1-s_{d+1}(k, \boldsymbol{g})-\sum_{j=d+1}^D s_j(k, \boldsymbol{g})r_j> \mathcal{D}_d\bigg(\frac{2^{d-1}}{k-B_d(k, \boldsymbol{g})} +s_{d+1}(k, \boldsymbol{g})\bigg)>0 \quad (d \in \Delta)$$
    and
    $$ t_0(k, \boldsymbol{g})=1-s_1(k, \boldsymbol{g})-\sum_{j=1}^Ds_j(k, \boldsymbol{g})r_j>0 .$$
Hence, by the conventions of $Q_j$ for $j \in \Delta$, \eqref{L bound in I1d}, \eqref{meas of AI1d} and calculations which are similar to \eqref{G bound cases}, we obtain
    \begin{equation*}
        \begin{split}
        \mathcal{J}(\mathcal{A}(L_0;I_d^{(1)}(k)))
        &\ll |\mathcal{A}(L_0;I_d^{(1)}(k))|P^{k+2t+2l}L_0(\log P)^{2t+2l}|\mathfrak{n}| \\
        &\ll P^{k+2t+2l+\varepsilon_0}Q_{d+1}\prod_{j=d+1}^D\bigg(\frac{Q_j}{P^j}\bigg)^{r_j}L_0|\mathfrak{n}| \\
        &\ll P^{k+2t+2l-\mathcal{D}+\mathcal{D}_d+\varepsilon_0  }L_0^{1-(s_{d+1}(k, \boldsymbol{g})+s_{d+1}(k, \boldsymbol{g})r_{d+1}+\cdots+s_{D}(k, \boldsymbol{g})r_{D})}|\mathfrak{n}| \\
        &\ll P^{k+2t+2l-\mathcal{D}+\mathcal{D}_d+\varepsilon_0 -\frac{1-(s_{d+1}(k, \boldsymbol{g})+s_{d+1}(k, \boldsymbol{g})r_{d+1}+\cdots+s_{D}(k, \boldsymbol{g})r_{D})}{2^{d-1}/(k-B_d(k, \boldsymbol{g}))+s_{d+1}(k, \boldsymbol{g})}}|\mathfrak{n}|
        \end{split}
    \end{equation*}
for every degree $d<D$. By dyadic argument and \eqref{def td},
\begin{equation}\label{J(I1d) bound}
\begin{split}
            \mathcal{J}(I_d^{(1)}(k))
            &\ll (\log P) \cdot\mathcal{J}(\mathcal{A}(L_0;I_d^{(1)}(k))) \\
            &\ll P^{k+2t+2l-\mathcal{D}+2\varepsilon_0-t_d(k, \boldsymbol{g})}|\mathfrak{n}|.
\end{split}
\end{equation}
After similar calculations for $\mathcal{J}(I^{(2)}(k))$, by \eqref{meas of AI2} and the conventions of $Q_j$, we get
    \begin{equation*}
        \begin{split}
            \mathcal{J}(I^{(2)}(k))
            &\ll (\log P) \cdot\mathcal{J}(\mathcal{A}(L_0;I^{(2)}(k))) \\
            &\ll |\mathcal{A}(L_0;I^{(2)}(k))|P^{k+2t+2l}L_0(\log P)^{2t+2l+1}|\mathfrak{n}|\\
            &\ll P^{k+2t+2l}L_0P^{\varepsilon_0}Q_{1}\prod_{j=1}^D\bigg(\frac{Q_j}{P^j}\bigg)^{r_j}(\log P)^{2t+2l+1}|\mathfrak{n}|\\
            &\ll P^{k+2t+2l+\varepsilon_0-\mathcal{D} }L_0^{1-(s_1(k, \boldsymbol{g})+s_1(k, \boldsymbol{g})r_1+\cdots+s_D(k, \boldsymbol{g})r_D)}|\mathfrak{n}|,
        \end{split}
    \end{equation*}
    which is, by \eqref{L bound in I2} and \eqref{def t0},
    \begin{equation}\label{J(I2) bound}
        \begin{split}
           \mathcal{J}(I^{(2)}(k)) 
           &\ll P^{k+2t+2l-\mathcal{D}+2\varepsilon_0 }X^{-\frac{1-s_1(k, \boldsymbol{g})-\sum_{j=1}^Ds_j(k, \boldsymbol{g})r_j}{4s_{1}(k, \boldsymbol{g})}} |\mathfrak{n}| \\ 
           &= P^{k+2t+2l-\mathcal{D}+2\varepsilon_0 }X^{-\frac{t_0(k, \boldsymbol{g})}{4s_{1}(k, \boldsymbol{g})}} |\mathfrak{n}|.
        \end{split}
    \end{equation}
Now, by \eqref{J(I1D) bound}, \eqref{J(I1d) bound} and \eqref{J(I2) bound}, we deduce that
    \begin{equation}\label{J(m) bound}
        \begin{split}
            \mathcal{J}(\mathfrak{m}(X))
            &\ll \sum_{d \in \Delta}|\mathcal{J}(I_d^{(1)}(k))|+|\mathcal{J}(I^{(2)}(k))|  \\
            &\ll P^{k+2t+2l-\mathcal{D}+2\varepsilon_0-\min_{d \in \Delta}{t_d(k, \boldsymbol{g})}}|\mathfrak{n}|+P^{k+2t+2l+2\varepsilon_0-\mathcal{D} }X^{-\frac{t_0(k, \boldsymbol{g})}{4s_{1}(k, \boldsymbol{g})}}|\mathfrak{n}|.
        \end{split}
    \end{equation}

We are left with $\mathcal{J}(\mathfrak{M}(X))$.
   By Cauchy's inequality in $\boldsymbol{\alpha}_2$ for fixed
$\boldsymbol{\alpha}_1-\boldsymbol{\alpha}_2$,
    \begin{align*}
        \mathcal{J}(\mathfrak{M}(X)) 
        &\ll  |\mathfrak{M}(X)|\sum_{\boldsymbol{x}\in \mathcal{B}_k(P)}\int_{\mathfrak{n}}|\mathcal{E}(\boldsymbol{\alpha};\boldsymbol{x})T(\boldsymbol{\alpha};\boldsymbol{x})|^2d\boldsymbol{\alpha} \\
        &\ll \sum_{q\le X}\sum_{\substack{1\le \boldsymbol{a}\le q \\ (a_1,\ldots,a_R,q)=1}}\prod_{d \in \Delta}\bigg( \frac{X}{qP^d}\bigg)^{r_d}\sum_{\boldsymbol{x}\in \mathcal{B}_k(P)}\int_{\mathfrak{n}}|\mathcal{E}(\boldsymbol{\alpha};\boldsymbol{x})T(\boldsymbol{\alpha};\boldsymbol{x})|^2d\boldsymbol{\alpha} \\
        &\ll X^{R+1}P^{-\mathcal{D} }\sum_{\boldsymbol{x}\in \mathcal{B}_k(P)}\int_{\mathfrak{n}}|\mathcal{E}(\boldsymbol{\alpha};\boldsymbol{x})T(\boldsymbol{\alpha};\boldsymbol{x})|^2d\boldsymbol{\alpha}.
    \end{align*}
Combining this with \eqref{J(m) bound} completes the proof.
\end{proof}

\section{A crucial proposition}
The purpose of this section is to establish Proposition \ref{prop},
which we shall use to bound the contribution from the minor arcs.

Let $\boldsymbol{F}=(F_1,\ldots,F_R)=(F_{i,d})_{\substack{d \in \Delta \\ 1\leqslant i \leqslant r_d}}$ 
with $F_{i,d}=F_{i,d}(x_1,\ldots,x_n)$ being forms of degree $d$ in $n$ variables for all $i, d$.
Write
$$\boldsymbol{x}=(\boldsymbol{y},\boldsymbol{z},\boldsymbol{w}),$$
where $\boldsymbol{y}\in \mathbb{N} ^m, \boldsymbol{z}\in \mathbb{N} ^{s}, \boldsymbol{w}\in \mathbb{N}^t$ and
$m+s+t=n.$ 
Then each $F_{i,d}$ can be decomposed as 
\begin{align}\label{decom F into fgh 1}
    F_{i,d}(\boldsymbol{y},\boldsymbol{z},\boldsymbol{w})=f_{i,d}(\boldsymbol{y})+g_{i,d}(\boldsymbol{y},\boldsymbol{z})
    +h_{i,d}(\boldsymbol{y},\boldsymbol{z},\boldsymbol{w}),
\end{align}
where
\begin{align*}
    \deg_{\boldsymbol{y}}(g_{i,d})<d \ , \ \deg_{(\boldsymbol{y},\boldsymbol{z})}(h_{i,d})<d, 
\end{align*}
and we remark that both $m$ and $s$ could be $0$. 
Put
\begin{align*}
    \boldsymbol{f}=(f_1,\ldots,f_R)=(f_{i,d})_{\substack{d\in \Delta \\ 1\le i\le r_d}}
   \end{align*}
and set $\boldsymbol{g}$ and $\boldsymbol{h}$ similarly.  
Then the  exponential sum defined in \eqref{def/SFa} can be written as  
\begin{align}\label{define SF with decom for F}
    S_{\boldsymbol{F}}(\boldsymbol{\alpha})=\sum_{\boldsymbol{y} \in \mathcal{B}_m(P)}\sum_{\boldsymbol{z} \in \mathcal{B}_s(P)}
    \sum_{\boldsymbol{w} \in \mathcal{B}_t(P)}\Lambda(\boldsymbol{y})
    \Lambda(\boldsymbol{z})\Lambda(\boldsymbol{w})e(\boldsymbol{\alpha}\cdot \boldsymbol{F}(\boldsymbol{y},\boldsymbol{z},\boldsymbol{w})). 
\end{align}
Also define 
\begin{align}\label{write R secttrans}
    \mathcal{E}_{\boldsymbol{y},\boldsymbol{z}}(\boldsymbol{\alpha}):=
    \sum_{\boldsymbol{w} \in \mathcal{B}_t(P)}\Lambda(\boldsymbol{w})e(\boldsymbol{\alpha}\cdot \boldsymbol{h}(\boldsymbol{y},\boldsymbol{z},\boldsymbol{w})).
\end{align}

\begin{proposition}\label{prop}
    Let $S_{\boldsymbol{F}}(\boldsymbol{\alpha})$ be as in \eqref{define SF with decom for F} with $\boldsymbol{F}=(F_1,\ldots,F_R)$ being decomposed as in \eqref{decom F into fgh 1}.
    Suppose $m$ is admissible for $\boldsymbol{f}$ and $m+s$ is admissible for $\boldsymbol{g}$. 
    Suppose that $\deg_{\boldsymbol{y}}g_{i,d}<C$ and
$\deg_{(\boldsymbol{y},\boldsymbol{z})}h_{i,d}<C$ for all $i,d$.
    Let $X\le Q$.
    Then we have
\begin{align*}
\int_{\mathfrak{n}}S_{\boldsymbol{F}}(\boldsymbol{\alpha})d\boldsymbol{\alpha} 
\ll
        & \, P^{n-\frac{1}{2}\mathcal{D} +\varepsilon-\frac{1}{2}\min_{d \in \Delta}{t_d(m, \boldsymbol{f})} } |\mathfrak{n}|^{\frac{1}{2}}  + P^{n-\frac{1}{2}\mathcal{D}+\varepsilon}X^{-\frac{t_0(m, \boldsymbol{f})}{8s_{1}(m, \boldsymbol{f})}} |\mathfrak{n}|^{\frac{1}{2}} \\
        & \ + P^{n-\frac{3}{4}\mathcal{D} +\varepsilon-\frac{1}{4}\min_{d \in \Delta}{t_d(m+s, \boldsymbol{g})} } X^{\frac{1}{2}R+\frac{1}{2}}|\mathfrak{n}|^{\frac{1}{4}} \\
        & \ + P^{n-\frac{3}{4}\mathcal{D}+\varepsilon}X^{\frac{1}{2}R+\frac{1}{2}-\frac{t_0(m+s, \boldsymbol{g})}{16s_{1}(m+s, \boldsymbol{g})}} |\mathfrak{n}|^{\frac{1}{4}} \\
        & \  +  P^{m+s-\mathcal{D}+\varepsilon }X^{R+1} \sup(\mathcal{E}),
    \end{align*}
    where 
    \begin{align*}
        \sup(\mathcal{E})=\sup_{\boldsymbol{\alpha}\in \mathfrak{n}}\sup_{\boldsymbol{y}}\sup_{\boldsymbol{z}}|\mathcal{E}_{\boldsymbol{y},\boldsymbol{z}}(\boldsymbol{\alpha})|.
    \end{align*}
\end{proposition}

\begin{proof}
The proof of this proposition is similar to that of \cite[Proposition 5.1]{LiuZha23} but more complicated, 
and so we write in full detail. Let
    \begin{align}\label{define T in sect5}
        T(\boldsymbol{\alpha};\boldsymbol{y})=
        \sum_{\boldsymbol{z}}\sum_{\boldsymbol{w}}\Lambda(\boldsymbol{z})\Lambda(\boldsymbol{w})e(\boldsymbol{\alpha}\cdot \boldsymbol{G}(\boldsymbol{y},\boldsymbol{z},\boldsymbol{w})),
    \end{align}
    where
    \begin{align}\label{define G in sect5}
        \boldsymbol{G}(\boldsymbol{y},\boldsymbol{z},\boldsymbol{w})=
        \boldsymbol{g}(\boldsymbol{y},\boldsymbol{z})+\boldsymbol{h}(\boldsymbol{y},\boldsymbol{z},\boldsymbol{w}).
    \end{align}
    By \eqref{decom F into fgh 1} and \eqref{define SF with decom for F}, we have
    \begin{align*}
        S_{\boldsymbol{F}}(\boldsymbol{\alpha})=\sum_{\boldsymbol{y}}\Lambda(\boldsymbol{y})
        e(\boldsymbol{\alpha} \cdot \boldsymbol{f}(\boldsymbol{y}))T(\boldsymbol{\alpha};\boldsymbol{y}), 
    \end{align*} 
    and therefore
    \begin{align*}
        \int_{\mathfrak{n}}S_{\boldsymbol{F}}(\boldsymbol{\alpha})d\boldsymbol{\alpha}
        &=\int_{\mathfrak{n}}\sum_{\boldsymbol{y}}\Lambda(\boldsymbol{y})
        e(\boldsymbol{\alpha} \cdot \boldsymbol{f}(\boldsymbol{y}))T(\boldsymbol{\alpha};\boldsymbol{y})d\boldsymbol{\alpha} \\
        &=\sum_{\boldsymbol{y}}\Lambda(\boldsymbol{y})\int_{\mathfrak{n}}
        e(\boldsymbol{\alpha} \cdot \boldsymbol{f}(\boldsymbol{y}))T(\boldsymbol{\alpha};\boldsymbol{y})d\boldsymbol{\alpha}.
    \end{align*}
    Cauchy's inequality now gives 
    \begin{align}\label{S square to I}
        \bigg|\int_{\mathfrak{n}}S_{\boldsymbol{F}}(\boldsymbol{\alpha})d\boldsymbol{\alpha}\bigg|^2 
        \ll P^m(\log P)^m \mathcal{I}_{\mathfrak{n}},
    \end{align}
where
\begin{align*}
        \mathcal{I}_{\mathfrak{n}}
        =\sum_{\boldsymbol{y}}
        \bigg|\int_{\mathfrak{n}}
        e(\boldsymbol{\alpha}\cdot \boldsymbol{f}(\boldsymbol{y}))
        T(\boldsymbol{\alpha};\boldsymbol{y})d\boldsymbol{\alpha}\bigg|^2.
\end{align*} 
By Lemma \ref{lemma mean J} with $k=m$ and $t=0$,  
\begin{equation}\label{I to T}
\begin{split}
        \mathcal{I}_{\mathfrak{n}}\ll 
        & \, P^{m+2l-\mathcal{D}+\varepsilon_1 -\min_{d \in \Delta}{t_d(m, \boldsymbol{f})} } | \mathfrak{n}|  + P^{m+2l-\mathcal{D}+\varepsilon_1}X^{-\frac{t_0(m, \boldsymbol{f})}{4s_{1}(m, \boldsymbol{f})}} | \mathfrak{n}| \\
        & \ + X^{R+1}P^{-\mathcal{D} }\mathcal{T}_{\mathfrak{n}},
\end{split}
\end{equation}
    where $l=s+t$ 
    and
    \begin{align}\label{define Tn}
        \mathcal{T}_{\mathfrak{n}} =\sum_{\boldsymbol{y}}\int_{\mathfrak{n}} 
        |T(\boldsymbol{\alpha};\boldsymbol{y})|^2 d\boldsymbol{\alpha}.
    \end{align}

Now we estimate $\mathcal{T}_{\mathfrak{n}}$. By \eqref{write R secttrans}, \eqref{define T in sect5} and \eqref{define G in sect5}, we have
    \begin{align*}
        T(\boldsymbol{\alpha};\boldsymbol{y})=\sum_{\boldsymbol{z}}\Lambda(\boldsymbol{z})e(\boldsymbol{\alpha} \cdot \boldsymbol{g}(\boldsymbol{y},\boldsymbol{z}))
        \mathcal{E}_{\boldsymbol{y},\boldsymbol{z}}(\boldsymbol{\alpha}), 
    \end{align*}
and therefore 
\begin{align*}
        \int_{\mathfrak{n}}|T(\boldsymbol{\alpha};\boldsymbol{y})|^2d\boldsymbol{\alpha}
        =&\int_{\mathfrak{n}}T(\boldsymbol{\alpha};\boldsymbol{y})T(-\boldsymbol{\alpha};\boldsymbol{y})d\boldsymbol{\alpha}\\ 
        = & \sum_{\boldsymbol{z}}\Lambda(\boldsymbol{z})\int_{\mathfrak{n}}e(\boldsymbol{\alpha} \cdot \boldsymbol{g}(\boldsymbol{y},\boldsymbol{z}))
        \mathcal{E}_{\boldsymbol{y},\boldsymbol{z}}(\boldsymbol{\alpha})
        T(-\boldsymbol{\alpha};\boldsymbol{y})d\boldsymbol{\alpha}.
\end{align*}
Another application of Cauchy's inequality then yields 
\begin{align}\label{T square to J pre}
        \bigg(\int_{\mathfrak{n}} |T(\boldsymbol{\alpha};\boldsymbol{y})|^2d\boldsymbol{\alpha}\bigg)^2 
        \ll P^s(\log P)^s\mathcal{J}_{\mathfrak{n},\boldsymbol{y}},
\end{align}
where
\begin{align*}
        \mathcal{J}_{\mathfrak{n},\boldsymbol{y}}=\sum_{\boldsymbol{z}}
        \bigg|\int_{\mathfrak{n}}e(\boldsymbol{\alpha}\cdot \boldsymbol{g}(\boldsymbol{y},\boldsymbol{z}))
        \mathcal{E}_{\boldsymbol{y},\boldsymbol{z}}(\boldsymbol{\alpha})
        T(-\boldsymbol{\alpha};\boldsymbol{y})d\boldsymbol{\alpha}\bigg|^2.
\end{align*}
Hence \eqref{define Tn}, \eqref{T square to J pre} and Cauchy's inequality give 
    \begin{align}\label{T square to J}
        \mathcal{T}_{\mathfrak{n}}^2\ll P^{m+s}(\log P)^s\mathcal{J}_{\mathfrak{n}},
    \end{align}
    where
    \begin{align*}
        \mathcal{J}_{\mathfrak{n}}=\sum_{\boldsymbol{y}}\mathcal{J}_{\mathfrak{n},\boldsymbol{y}}.
    \end{align*}
    On applying Lemma \ref{lemma mean J} with $k=m+s$, $l=s+t$, we get
    \begin{equation}\label{J to K}
        \begin{split}
        \mathcal{J}_{\mathfrak{n}}\ll & \, P^{n+t+2l-\mathcal{D}+\varepsilon_1-\min_{d \in \Delta}{t_d(m+s, \boldsymbol{g})} } | \mathfrak{n}|  + P^{n+t+2l-\mathcal{D}+\varepsilon_1}X^{-\frac{t_0(m+s, \boldsymbol{g})}{4s_{1}(m+s, \boldsymbol{g})}} | \mathfrak{n}| \\
        & \ + X^{R+1}P^{-\mathcal{D} }\mathcal{K}_{\mathfrak{n}},
        \end{split}
    \end{equation}
    where 
    \begin{align*}
        \mathcal{K}_{\mathfrak{n}}=\sum_{\boldsymbol{y}}\sum_{\boldsymbol{z}}\int_{\mathfrak{n}}|T(\boldsymbol{\alpha};\boldsymbol{y})|^2
        |\mathcal{E}_{\boldsymbol{y},\boldsymbol{z}}(\boldsymbol{\alpha})|^2d\boldsymbol{\alpha}.
    \end{align*}
    Recalling the definition of $\sup(\mathcal{E})$ and estimating elementarily, we have
    \begin{align*}
        \sum_{\boldsymbol{z}}|\mathcal{E}_{\boldsymbol{y},\boldsymbol{z}}(\boldsymbol{\alpha})|^2\le P^{s}\sup(\mathcal{E})^2,
    \end{align*}
and hence 
\begin{align}\label{K bound}
        \mathcal{K}_{\mathfrak{n}}\le \mathcal{T}_{\mathfrak{n}}\, P^{s}\sup(\mathcal{E})^2.
\end{align}
Inserting \eqref{K bound} into \eqref{J to K} gives 
\begin{align*}
        \mathcal{J}_{\mathfrak{n}}\ll 
        & \, P^{n+t+2l-\mathcal{D} +\varepsilon_1-\min_{d \in \Delta}{t_d(m+s, \boldsymbol{g})} } | \mathfrak{n}|  + P^{n+t+2l-\mathcal{D}+\varepsilon_1}X^{-\frac{t_0(m+s, \boldsymbol{g})}{4s_{1}(m+s, \boldsymbol{g})}} | \mathfrak{n}| \\
        & \  + \mathcal{T}_{\mathfrak{n}}\, P^{s-\mathcal{D} }X^{R+1} \sup(\mathcal{E})^2, 
\end{align*}
which in combination with \eqref{T square to J} gives  
\begin{align*}
        \mathcal{T}_{\mathfrak{n}} ^2\ll 
        & \, P^{2n+2l-\mathcal{D}+\varepsilon_1 -\min_{d \in \Delta}{t_d(m+s, \boldsymbol{g})} } | \mathfrak{n}|  + P^{2n+2l-\mathcal{D}+\varepsilon_1}X^{-\frac{t_0(m+s, \boldsymbol{g})}{4s_{1}(m+s, \boldsymbol{g})}} | \mathfrak{n}| \\
        & \  + \mathcal{T}_{\mathfrak{n}}\, P^{m+2s+\varepsilon_1-\mathcal{D} }X^{R+1} \sup(\mathcal{E})^2.
    \end{align*}
It follows that 
\begin{equation}\label{T bound}
\begin{split}
        \mathcal{T}_{\mathfrak{n}} \ll 
        & \, P^{n+l-\frac{1}{2}\mathcal{D}+\frac{1}{2}\varepsilon_1 -\frac{1}{2}\min_{d \in \Delta}{t_d(m+s, \boldsymbol{g})} } | \mathfrak{n}|^{\frac{1}{2}}  + P^{n+l-\frac{1}{2}\mathcal{D}+\frac{1}{2}\varepsilon_1}X^{-\frac{t_0(m+s, \boldsymbol{g})}{8s_{1}(m+s, \boldsymbol{g})}} | \mathfrak{n}|^{\frac{1}{2}} \\
        & \  +  P^{m+2s+\varepsilon_1-\mathcal{D} }X^{R+1} \sup(\mathcal{E})^2.
\end{split}
\end{equation}

Combining \eqref{T bound} and \eqref{I to T}, we deduce 
    \begin{align*}
        \mathcal{I}_{\mathfrak{n}}\ll 
        & \, P^{m+2l-\mathcal{D}+\varepsilon_1 -\min_{d \in \Delta}{t_d(m, \boldsymbol{f})} } | \mathfrak{n}|  + P^{m+2l-\mathcal{D}+\varepsilon_1}X^{-\frac{t_0(m, \boldsymbol{f})}{4s_{1}(m, \boldsymbol{f})}} | \mathfrak{n}| \\
        & \ + P^{n+l-\frac{3}{2}\mathcal{D} +\frac{1}{2}\varepsilon_1-\frac{1}{2}\min_{d \in \Delta}{t_d(m+s, \boldsymbol{g})} } X^{R+1}| \mathfrak{n}|^{\frac{1}{2}}  + P^{n+l-\frac{3}{2}\mathcal{D}+\frac{1}{2}\varepsilon_1}X^{R+1-\frac{t_0(m+s, \boldsymbol{g})}{8s_{1}(m+s, \boldsymbol{g})}} | \mathfrak{n}|^{\frac{1}{2}} \\
        & \  +  P^{m+2s+\varepsilon_1-2\mathcal{D} } X^{2R+2}\sup(\mathcal{E})^2.
    \end{align*}
Inserting this into \eqref{S square to I} and noting that $n=m+s+t=m+l$, we conclude that 
\begin{align*}
        \bigg|\int_{\mathfrak{n}}S_{\boldsymbol{F}}(\boldsymbol{\alpha})d\boldsymbol{\alpha}\bigg|^2 \ll
        & \, P^{2n-\mathcal{D} +\varepsilon -\min_{d \in \Delta}{t_d(m, \boldsymbol{f})} }| \mathfrak{n}|  + P^{2n-\mathcal{D}+\varepsilon}X^{-\frac{t_0(m, \boldsymbol{f})}{4s_{1}(m, \boldsymbol{f})}} | \mathfrak{n}| \\
        & \ + P^{2n-\frac{3}{2}\mathcal{D} +\varepsilon-\frac{1}{2}\min_{d \in \Delta}{t_d(m+s, \boldsymbol{g})} } X^{R+1}| \mathfrak{n}|^{\frac{1}{2}}  \\
        & \ + P^{2n-\frac{3}{2}\mathcal{D}+\varepsilon}X^{R+1-\frac{t_0(m+s, \boldsymbol{g})}{8s_{1}(m+s, \boldsymbol{g})}} | \mathfrak{n}|^{\frac{1}{2}} \\
        & \  +  P^{2m+2s-2\mathcal{D}+\varepsilon }X^{2R+2} \sup(\mathcal{E})^2.
\end{align*}
This proves the proposition. 
\end{proof}

\begin{remark}
Proposition \ref{prop} is still true if in \eqref{define SF with decom for F} 
we erase the three weights  
$\Lambda(\boldsymbol{y})$, 
$\Lambda(\boldsymbol{z})$ and 
$\Lambda(\boldsymbol{w})$. 
This is clear from the proof. 
\end{remark}

One sees from Proposition \ref{prop} that, if one has a nontrivial bound
$$
\sup(\mathcal{E}) \ll P^t Q^{-\omega_{\Delta,R}} 
$$
for some $\omega_{\Delta,R}>0$, then it is possible to derive a nice upper bound for 
$\int_{\mathfrak{n}}S_{\boldsymbol{F}}(\boldsymbol{\alpha})d \boldsymbol{\alpha}$ 
provided that all $t_d(m, \boldsymbol{f})$ and $t_d(m+s, \boldsymbol{g})$ 
are large enough, and $Q$ is chosen appropriately.

\section{Sums over primes}  
In this section we quote a lemma from \cite{LiuZha23} on exponential sums over primes. 
 
For $(H_1,\ldots,H_R)=(H_{i,d})_{\substack{d \in \Delta \\1\leqslant i \le r_d}}\in \Z[x_1,\ldots,x_n]$, we set
$$\rank(\boldsymbol{H})=\rank(H_1,\ldots,H_R)$$
by viewing $H_1,\ldots,H_R$ as vectors in the linear space $\R[x_1,\ldots,x_n]$ over $\R$, 
i.e., the dimension of the linear subspace of $\R[x_1,\ldots,x_n]$ generated by $H_1,\ldots,H_R$. 
For each $i$ and $d$, let $H_{i,d}(x_1,\ldots,x_n)$ be a form of degree $d$ and let  $g_{i,d}(x_1,\ldots,x_n)$ be a polynomial of $\boldsymbol{x}=(x_1,\ldots,x_n)$ of lower degree. 
Letting 
$$
F_{i,d}=H_{i,d}+g_{i,d} \quad (d \in \Delta , \ 1\le i\le r_d), 
$$
we define the exponential sum 
\begin{align}\label{define E}
    \mathcal{E}(\boldsymbol{\alpha})=\sum_{(x_1,\ldots,x_n)\in \mathcal{B}_n(P)}\Lambda(x_1)\cdots \Lambda(x_n)
    e(\boldsymbol{\alpha}\cdot \boldsymbol{F}(x_1,\ldots,x_n)).
\end{align}

Note that the method of \cite[Lemma 6.9]{LiuZha23} still works for differing degrees.  
Therefore we have the following result. 

\begin{lemma}\label{lemma bound E1}
Let $\mathcal{E}(\boldsymbol{\alpha})$ be as in \eqref{define E}. 
In each $H_{i,d}$, omit a monomial $M$ if some $F_{j,e}$ with
$e>d$ has total degree at least $d$ in the variables of $M$.
Suppose that these omissions preserve the rank of $\boldsymbol H$.
    Suppose $\rank(\boldsymbol{H})=R$. Let $Q=P^{\varpi}$ with $0<\varpi<\frac{1}{4}$ and $\boldsymbol{\alpha}\in \mathfrak{m}=\mathfrak{m}(Q)$. 
    Then we have
    \begin{align*}
        \mathcal{E}(\boldsymbol{\alpha})\ll P^n Q^{-\frac{1}{2^{D}R}+\varepsilon}.
    \end{align*}
\end{lemma}

\begin{proof}
As in the proof of \cite[Lemma~6.9]{LiuZha23}, full rank gives
a retained monomial whose coefficient in $\boldsymbol\alpha\cdot\boldsymbol F$
lies on the one-dimensional minor arcs and which remains of highest
degree with the same coefficient when the other variables are fixed.
\end{proof}

\section{Contribution from the minor arcs}
This section is devoted to treating the contribution from the minor arcs. We begin with numerical estimates for the quantities 
$$
B_d(n, \boldsymbol{F}), s_d(n, \boldsymbol{F}), t_d(n, \boldsymbol{F}) \quad (d\in \Delta, d=0). 
$$ 
Define the {\it singular loci} of the system $\boldsymbol{F}=(F_1,\ldots,F_R)$ as
\begin{equation*}
    V^{\ast}_{\boldsymbol{F}}(n)=\{\boldsymbol{x}\in\mathbb{A} ^{n}:\rank(J_{\boldsymbol{F}}(\boldsymbol{x}))<R\}
\end{equation*}
in the sense of Birch. 
Then it is clear that, by the definition of $B_d(n, \boldsymbol{F})$ in \eqref{def Bd},
\begin{equation}\label{compare dim}
    \dim V^{\ast}_{\boldsymbol{F}}(n) \ge B_d(n, \boldsymbol{F}) \quad (d \in \Delta).
\end{equation}
Also it is easy to see
\begin{equation*}
    \dim V^{\ast}_{\boldsymbol{F}}(n)\le R
\end{equation*}
for a nonsingular system $\boldsymbol{F}$.
Write
\begin{align}\label{def ud}
    u_d:=\sum^D_{i=d}2^{i-1}(i-1)r_i \quad (1\le d\le D).
\end{align}

\begin{lemma}\label{Lem/7/1}
    Let $s_d(n, \boldsymbol{F})$ be as in \eqref{def sd} for all $d$. Then 
    \begin{align*}
        s_C(n, \boldsymbol{F})=s_1(n, \boldsymbol{F})\le A_1(n, \boldsymbol{F}),
    \end{align*}
    where
    \begin{equation}\label{def A1n}
        A_1(n, \boldsymbol{F}):=\frac{2^{D-1}(D-1)R}{n-\dim V^{\ast}_{\boldsymbol{F}}(n)}.
    \end{equation}   
\end{lemma}

\begin{proof}
By \eqref{compare dim} and elementary argument. 
\end{proof}
It follows that 
    \begin{equation}\label{t_0 lower bound}
        t_0(n, \boldsymbol{F})=1-s_1(n, \boldsymbol{F})-\sum^D_{j=1}s_j(n, \boldsymbol{F})r_j \ge 1-A_1(n, \boldsymbol{F})(R+1).
    \end{equation}

\begin{lemma}\label{Lem/7/2}
    Let $t_d(n, \boldsymbol{F})$ be as in \eqref{def td} for $d \in \Delta$.
    If \[
n-\dim V_{\boldsymbol F}^{*}(n)
\ge \max\{2^{D-1}D,\;2^{D-1}(D-1)R(R+1)\}.
\] then
    \begin{align*}
        \min_{d\in \Delta}{t_d(n, \boldsymbol{F})}\ge A_2(n, \boldsymbol{F}),
    \end{align*}
    where
    \begin{equation*}
        A_2(n, \boldsymbol{F}):=\frac{n-\dim V^{\ast}_{\boldsymbol{F}}(n)-2^{D-1}(D-1)R(R+1)}{2^{D-1}+2^{D-1}(D-1)R}-\mathcal{D}+D.
    \end{equation*}
\end{lemma}
\begin{proof}
By \eqref{compare dim}, we have
    \begin{align*}
        t_D(n, \boldsymbol{F})=\frac{n-B_D(n, \boldsymbol{F})}{2^{D-1}}-\mathcal{D}\ge \frac{n-\dim V^{\ast}_{\boldsymbol{F}}(n)}{2^{D-1}}-\mathcal{D}. 
    \end{align*}
While for $d\in \Delta$ but $d <D $, we deduce by \eqref{def ud} that 
\begin{align*}
        t_d(n, \boldsymbol{F}) 
        &\ge \frac{n-\dim V^{\ast}_{\boldsymbol{F}}(n)-u_{d+1}-\sum^D_{j=d+1}u_jr_j}{2^{d-1}+u_{d+1}}-\mathcal{D}_d \\
        &\ge \frac{n-\dim V^{\ast}_{\boldsymbol{F}}(n)-2^{D-1}(D-1)R(R+1)}{2^{D-1}+2^{D-1}(D-1)R}-\mathcal{D}+D.
\end{align*}
Hence if $n-\dim V_{\boldsymbol F}^{*}(n)\ge \max\{2^{D-1}D,\;2^{D-1}(D-1)R(R+1)\}$ then, by elementary computations,  
    $$\frac{n-\dim V^{\ast}_{\boldsymbol{F}}(n)}{2^{D-1}}-\mathcal{D} \ge \frac{n-\dim V^{\ast}_{\boldsymbol{F}}(n)-2^{D-1}(D-1)R(R+1)}{2^{D-1}+2^{D-1}(D-1)R}-\mathcal{D}+D.$$
This completes the proof.
\end{proof}

After these preparations, we can finally analyze the contribution from minor arcs precisely. 
Recall that each $F_{i,d}$ is decomposed as in \eqref{decom F into fgh 1}. 
Then each $h_{i,d}$ can be uniquely decomposed as
\begin{align}\label{decompose h into GH}
    h_{i,d}(\boldsymbol{y},\boldsymbol{z},\boldsymbol{w})=G_{i,d}(\boldsymbol{y},\boldsymbol{z},\boldsymbol{w})
    +H_{i,d}(\boldsymbol{w}),
\end{align}  
where $\deg_{\boldsymbol{w}}(G_{i,d})<d$ and $H_{i,d}$ is a form in $\boldsymbol{w}$ with degree $d$.
Write 
\begin{align*}
 \boldsymbol{H}=(H_1,\ldots,H_R)=(H_{i,d})_{\substack{d\in \Delta \\ 1\le i\le r_d}}.
\end{align*}

\begin{lemma}\label{lem/min/arcs/est}
    Let $\boldsymbol{F}=(F_1,\ldots,F_R)=(F_{i,d})_{\substack{d \in \Delta \\ 1\le i\le r_d}}$ be decomposed as in \eqref{decom F into fgh 1} and \eqref{decompose h into GH}. 
    Let $S_{\boldsymbol{F}}(\boldsymbol{\alpha})$ be as in \eqref{def/SFa}.
    Let $Q=P^{\varpi}$ with $0<\varpi<\frac{1}{4}$ and $\mathfrak{m}=\mathfrak{m}(Q) $ be as in \eqref{define mQ}.
    Assume the degree bounds of Proposition~\ref{prop}.
For each fixed $\boldsymbol y,\boldsymbol z$, suppose that the omissions
of Lemma~\ref{lemma bound E1} for
$(h_{i,d}(\boldsymbol y,\boldsymbol z,\cdot))_{i,d}$
preserve the rank of $\boldsymbol H$.
    Assume that \text{\rm (i)}
    \begin{equation}\label{def kappa1}
        \begin{split}
            & \ \ m-\dim V^{\ast}_{\boldsymbol{f}}(m)-2^{D-1}(D-1)R(R+1) \\
            & \ge [(R+1)\varpi+\mathcal{D}-D](2^{D-1}+2^{D-1}(D-1)R)+1,
        \end{split}
    \end{equation}
    \text{\rm (ii)}
    \begin{equation}\label{def kappa1'}
        \begin{split}
            m-\dim V^{\ast}_{\boldsymbol{f}}(m)-2^{D-1}(D-1)R(R+1) \ge 2^{2D+2}(D-1)R^2(R+1)^2+1,
        \end{split}
    \end{equation}
    \text{\rm (iii)}
    \begin{equation}\label{def kappa2}
        \begin{split}
            & \ \ m+s-\dim V^{\ast}_{\boldsymbol{g}}(m+s)-2^{D-1}(D-1)R(R+1) \\
            &\ge \bigg[\bigg(R+1+\frac{1}{2^DR}\bigg)\varpi+\mathcal{D}-D\bigg] (2^{D-1}+2^{D-1}(D-1)R)+1,
        \end{split}
    \end{equation}
    \text{\rm (iv)}
    \begin{equation}\label{def kappa2'}
        \begin{split}
            & \ \ m+s-\dim V^{\ast}_{\boldsymbol{g}}(m+s)-2^{D-1}(D-1)R(R+1) \\
            &\ge [8R+8+2^{D+3}R(R+1)^2]2^{D-1}(D-1)R+1,
        \end{split}
    \end{equation}
and \text{\rm (v)}
\begin{equation*}
        \rank(\boldsymbol{H})=R.
\end{equation*}
Let $\mathfrak n\subseteq\mathfrak m(Q)$ be measurable with
$|\mathfrak n|\ll P^{-\mathcal D}Q^{R+1}$.
Then there exists a constant $\delta=\delta_{\Delta,R}>0$ such that
\begin{align*}
\int_{\mathfrak n}S_{\boldsymbol{F}}(\boldsymbol{\alpha})d\boldsymbol{\alpha} \ll P^{n-\mathcal{D}}Q^{-\delta}.
\end{align*}
\end{lemma}
\begin{proof}
Recall that $t_0(n, \boldsymbol{F})>0$ and $t_d(n, \boldsymbol{F})>0$ for each $d \in \Delta$ is equivalent to $n$ is admissible for $\boldsymbol{F}$.

Now, we claim \eqref{def kappa1}, \eqref{def kappa1'}, \eqref{def kappa2} and \eqref{def kappa2'} guarantee that $m$ is 
admissible for $\boldsymbol{f}$ and $m+s$ is admissible for $\boldsymbol{g}$. This claim will be proved later. 
On this claim,  
we apply Proposition \ref{prop} and Lemma \ref{lemma bound E1}  to get
    \begin{align*}
        \int_{\mathfrak{n}}S_{\boldsymbol{F}}(\boldsymbol{\alpha})d\boldsymbol{\alpha} 
        \ll& P^{n-\mathcal{D} +\varepsilon-\frac{1}{2}\min_{d \in \Delta}{t_d(m, \boldsymbol{f})} }Q^{\frac{R+1}{2}}  + P^{n-\mathcal{D}+\varepsilon}X^{-\frac{t_0(m, \boldsymbol{f})}{8s_{1}(m, \boldsymbol{f})}}Q^{\frac{R+1}{2}}  \\
        &  + P^{n-\mathcal{D} +\varepsilon-\frac{1}{4}\min_{d \in \Delta}{t_d(m+s, \boldsymbol{g})} }X^{\frac{R+1}{2}} Q^{\frac{R+1}{4}} \\
        &  + P^{n-\mathcal{D}+\varepsilon}X^{\frac{R+1}{2}-\frac{t_0(m+s, \boldsymbol{g})}{16s_{1}(m+s, \boldsymbol{g})}}Q^{\frac{R+1}{4}}  \\
        &  +  P^{n-\mathcal{D}+\varepsilon }X^{R+1} Q^{-\frac{1}{2^{D}R}}. 
    \end{align*}
On choosing
    $$
    X=Q^{\frac{1}{2^{D+1}R(R+1)}},
    $$
the above is
\begin{equation}\label{Int/Sn/F}
    \begin{aligned}
        \int_{\mathfrak{n}}S_{\boldsymbol{F}}(\boldsymbol{\alpha})d\boldsymbol{\alpha} 
        \ll& P^{n-\mathcal{D} +\varepsilon}Q^{\frac{R+1}{2}-\frac{1}{2\varpi}\min_{d \in \Delta}{t_d(m, \boldsymbol{f})} } + P^{n-\mathcal{D}+\varepsilon}Q^{\frac{R+1}{2}-\frac{t_0(m, \boldsymbol{f})}{2^{D+4}R(R+1)s_{1}(m, \boldsymbol{f})}}  \\
        &  + P^{n-\mathcal{D} +\varepsilon }Q^{\frac{R+1}{4}+\frac{1}{2^{D+2}R}-\frac{1}{4\varpi}\min_{d \in \Delta}{t_d(m+s, \boldsymbol{g})}} \\
        &  + P^{n-\mathcal{D}+\varepsilon}Q^{\frac{R+1}{4}+\frac{1}{2^{D+2}R}-\frac{t_0(m+s, \boldsymbol{g})}{2^{D+5}R(R+1)s_{1}(m+s, \boldsymbol{g})}}  \\
        &  +  P^{n-\mathcal{D}+\varepsilon } Q^{-\frac{1}{2^{D+1}R}}.
    \end{aligned}
\end{equation}

Now we compute the exponents in \eqref{Int/Sn/F}. 
It follows from \eqref{def kappa1}, \eqref{def kappa1'}, \eqref{def kappa2} and \eqref{def kappa2'} respectively that
    $$A_2(m, \boldsymbol{f})\ge (R+1)\varpi+\frac{1}{2^{D-1}+2^{D-1}(D-1)R},$$
    $$\frac{1-A_1(m, \boldsymbol{f})(R+1)}{A_1(m, \boldsymbol{f})}\ge 2^{D+3}R(R+1)^2+\frac{1}{2^{D-1}(D-1)R},$$
    \begin{align*}
        A_2(m+s, \boldsymbol{g})\ge &\bigg(R+1+\frac{1}{2^DR}\bigg)\varpi +\frac{1}{2^{D-1}+2^{D-1}(D-1)R}
    \end{align*}
    and
    \begin{align*}
        \frac{1-A_1(m+s, \boldsymbol{g})(R+1)}{A_1(m+s, \boldsymbol{g})}\ge &8R+8+2^{D+3}R(R+1)^2 +\frac{1}{2^{D-1}(D-1)R}.
    \end{align*}
From these as well as Lemmas \ref{Lem/7/1} and \ref{Lem/7/2},
we deduce that 
$$
\frac{R+1}{2}-\frac{1}{2\varpi}\min_{d \in \Delta}{t_d(m, \boldsymbol{f})}<0, 
$$
$$
\frac{R+1}{2}-\frac{t_0(m, \boldsymbol{f})}{2^{D+4}R(R+1)s_{1}(m, \boldsymbol{f})}<0, 
$$
$$
\frac{R+1}{4}+\frac{1}{2^{D+2}R}-\frac{1}{4\varpi}\min_{d \in \Delta}{t_d(m+s, \boldsymbol{g})}<0 
$$
and
$$
\frac{R+1}{4}+\frac{1}{2^{D+2}R}-\frac{t_0(m+s, \boldsymbol{g})}{2^{D+5}R(R+1)s_{1}(m+s, \boldsymbol{g})}<0. 
$$
Inserting these four formulae into \eqref{Int/Sn/F} proves the lemma. 

We also remark that  
the above four inequalities additionally justify the earlier claim about the admissibility. 
The proof is therefore complete. 
\end{proof}

Define
\begin{align}\label{def iota1}
    \begin{split}
        \iota_1:=
        &[\mathcal{D}-D+1+2^{D+3}R^2(R+1) ](R+1)2^{D-1}(D-1)+R+2^{D-1}(D-1)R(R+1)
    \end{split}
\end{align}
and
\begin{align}\label{def iota2}
    \begin{split}
        \iota_2:=
        &[\mathcal{D}-D+2^{D+3}R^2(R+1)+8R ](R+1)2^{D-1}(D-1)  \\
        &+R+2^{D-1}(D-1)R(R+1).
    \end{split}
\end{align}
Then $m-\dim V^{\ast}_{\boldsymbol{f}}(m)\ge \iota_1$ (resp. $m+s-\dim V^{\ast}_{\boldsymbol{g}}(m+s)\ge \iota_2$) implies \eqref{def kappa1} and \eqref{def kappa1'} (resp. \eqref{def kappa2} and \eqref{def kappa2'}) easily.
In other words, $m-\dim V^{\ast}_{\boldsymbol{f}}(m)\ge \iota_1$ (resp. $m+s-\dim V^{\ast}_{\boldsymbol{g}}(m+s)\ge \iota_2$) implies $m$ is admissible for $\boldsymbol{f}$ (resp. $m+s$ is admissible for $\boldsymbol{g}$).

Denote by $\operatorname{codim} V_{\boldsymbol{F}}^\ast(n)$ the codimension of the singular loci, where $\boldsymbol{F}=\boldsymbol{F}(x_1,\ldots,x_n)$ is a system of forms in $n$ variables and 
$\operatorname{codim} V_{\boldsymbol{F}}^\ast(n)=n-\dim V_{\boldsymbol{F}}^\ast(n).$
Define
\begin{align}\label{def iota3}
        \iota_3:=R\iota_2+\iota_1+DR^3+2R^2+R,
\end{align}
where $D$ is as in \eqref{def/C/D}.
Note that
\begin{align*}
    D^24^{D+2}R^5\ge \iota_3 
\end{align*}
since $D\ge3$ and $R\ge2$. 

\begin{lemma}\label{lemma key}
    Let $\boldsymbol{F}=(F_1,\ldots,F_R)=(F_{i,d})_{\substack{d\in \Delta \\ 1\le i\le r_d }}$ with each $F_{i,d}$ being a form of degree $d$ in $n$ variables. 
    Let $\iota_1$ and $\iota_2$ be as in \eqref{def iota1} and \eqref{def iota2}. 
    Suppose that $\boldsymbol{F}$ is a nonsingular system and
    \begin{align*}
        n\ge D^24^{D+2}R^5.
    \end{align*}
    Then up to a permutation of variables, $\boldsymbol{F}$ can be decomposed as in \eqref{decom F into fgh 1} and \eqref{decompose h into GH} such that {\rm (i)} $\operatorname{codim} V_{\boldsymbol{f}}^\ast(m) \ge \iota_1$,  {\rm (ii)} $\operatorname{codim} V_{\boldsymbol{g}}^\ast(m+s) \ge \iota_2$
      and {\rm (iii)} $\rank(\boldsymbol{H})=R$.
\end{lemma}
\begin{proof}

As in \cite[Lemma 8.2]{LiuZha23}, choose $r_d$ independent
coefficient columns in each degree and include the variables of their
monomials among the last $t=DR$ variables. Thus $\rank(\boldsymbol H)=R$.
Take $s=\iota_2+t+2R-1$ and $m=n-s-t$.

The incidence $\boldsymbol\lambda J_{\boldsymbol F}(\boldsymbol x)=0$,
$[\boldsymbol\lambda]\in\mathbb P_{\mathbb C}^{R-1}$, has dimension at most $2R-1$
since $\dim V_{\boldsymbol F}^\ast(n)\le R$. On any coordinate subspace,
omitting $u$ columns raises this bound by at most $u$: each component
meets $\boldsymbol x=0$ under the scaling
$(\boldsymbol x,[\lambda_{i,d}])\mapsto
(c\boldsymbol x,[c^{D-d}\lambda_{i,d}])$, where all omitted equations
vanish. Keeping the $\boldsymbol y$-columns at $\boldsymbol z=\boldsymbol w=0$
and the $\boldsymbol z$-columns at $\boldsymbol w=0$ gives
codimensions at least $n-2(s+t)-2R+1\ge\iota_1$ and
$s-t-2R+1=\iota_2$, respectively, by \eqref{def iota3}.
\end{proof}

If the decomposition in Lemma~\ref{lemma key} satisfies the hypotheses of Lemma~\ref{lem/min/arcs/est}, then \eqref{Int/Sn/F} applies with $\mathfrak n=\mathfrak M(2Q)\setminus\mathfrak M(Q)$.
We now estimate the integral over $\mathfrak m(Q)$ in the general case.

\begin{lemma}\label{lemma over minor2}
Let $\boldsymbol{F}=(F_1,\ldots,F_R)$ be a system of forms with degree set $\Delta$ in $n$ variables. 
Let $D$ be as in \eqref{def/C/D} and $S_{\boldsymbol{F}}(\boldsymbol{\alpha})$ be as in \eqref{def/SFa}.
Let $Q=P^{\varpi}$ with $0<\varpi<\frac{1}{4}$ and $\mathfrak{m}=\mathfrak{m}(Q)$ be as in \eqref{define mQ}.
    Suppose that $\boldsymbol{F}$ is a nonsingular system and 
\begin{align}\label{Lem75/n>}
n\ge D^24^{D+2}R^5.
\end{align}
    Then there exists a constant $\delta=\delta_{\Delta,R}>0$ such that
    \begin{align*}
        \int_{\mathfrak{m}}S_{\boldsymbol{F}}(\boldsymbol{\alpha})d\boldsymbol{\alpha} \ll P^{n-\mathcal{D} }Q^{-\delta},
    \end{align*}
    where $\mathcal{D}$ is as in \eqref{def/R/calD}.
\end{lemma}

\begin{proof}
Let $Q\le Y\le P$, $\boldsymbol\alpha\in\mathfrak m(Y)$,
$V=Y^{1/(9D(R+1))}$ and $E=\mathcal D+1$. Vaughan's identity
\cite[proof of Lemma~6.8]{LiuZha23}, applied coordinatewise, gives
$n_1$ products $m_i y_i$ with $m_i\le V$ and $n_2=n-n_1$ products
$h_i j_i$ of dyadic lengths $H_i\ge J_i\gg V$, $H_iJ_i\asymp P$.
Partial summation and Perron's formula remove the logarithmic
weights and separate the product constraints at a cost of $P^\varepsilon$.

If $n_1\ge n/8$, Cauchy's inequality in the latter products
$\boldsymbol z$ gives
$\boldsymbol G=\boldsymbol F(\boldsymbol m\boldsymbol y,\boldsymbol z)
-\boldsymbol F(\boldsymbol m\boldsymbol y',\boldsymbol z)$.
On $\boldsymbol y=\boldsymbol y'$, the argument of Lemma~\ref{lemma key}
gives $n+n_1-B_d(n+n_1,\boldsymbol G)\ge n_1-2R+2$, hence
$s_1(n+n_1,\boldsymbol G)\le1/(12E)$. Lemma~\ref{lemma alternative}
at scale $P/M$, $M=\max m_i$, with $B\ll M^D$, saves $Y^{-2E}$.

Otherwise Cauchy's inequality twice gives
$\boldsymbol L_{\boldsymbol c}=\sum_{\mu,\nu=0}^1(-1)^{\mu+\nu}
\boldsymbol F(\boldsymbol c,\boldsymbol h_\mu\boldsymbol j_\nu)$,
with the first products fixed at $\boldsymbol c$.
On $\boldsymbol h_1=\boldsymbol j_1=0$, $r$ zero pairs leave $n_1+r$
omitted derivatives and fibres of dimension at most $n_2-r$.
Thus $4n_2-B_{2d}(4n_2,\boldsymbol L_{\boldsymbol0})\ge n/2$.

Suppose $|\sum_{\boldsymbol h_0,\boldsymbol h_1,\boldsymbol j_0,\boldsymbol j_1}
e(\boldsymbol\alpha\cdot\boldsymbol L_{\boldsymbol c})|
>(\prod_i H_iJ_i)^2Y^{-4E}$. Descend in $d$ with $q\le W$ and
$\|q\alpha_{i,e}\|\le WP^{-e}$ for $e>d$, initially $q=W=1$.
Put $J=\min J_i$, $Z=V^{1/(2r_d)}$, $A=V^{d-1/2}$.
Double $k,\ell$ together from $1$, stopping at $k\asymp P/(WZJ)$
and $\ell\asymp P/J^2$, respectively. With $\ell_i=\min(k,\ell,J_i/J)$,
take independent long shifts $qku_i/\ell_i$, $|u_i|\le Z$, and use
progressions of steps $\ell/\ell_i,\ell_i$ at scale $J$.
Since $WkZ/\ell_i\ll H_i$, partial summation removes the higher-degree
errors. Writing $N=k\ell^{d-1}$, the diagonal polar columns are the
two gradient blocks of $\boldsymbol L_{\boldsymbol0}$ in degree $2d$,
scaled by $qN(d-1)$ and $qNd$ up to a common factor.
Applying \cite[Lemma~4.1]{BroHB} on boxes of side $J$ with $K=J/Z$
after the long difference, and using H\"older's inequality over the shifts,
we obtain the polar inequalities for a proportion
$\gg Y^{-2^{2d+1}E}(\log P)^{-4n_2}$ of the tuples.
By the diagonal argument of \cite[Lemma~5.1]{BroHB} and the rank
bound above, the rank-deficient tuples have proportion $O(Z^{-n/2})$,
which is smaller by \eqref{Lem75/n>}. Hence a nonzero minor gives
$1\le v\ll A$ with $\|vNq\alpha_{i,d}\|\ll A/(ZJ^{2d-1})$ for all $i$.
As $A^2\ll J^{2d-1}$, successive approximants agree after scaling to
$N\asymp P^d/(WZJ^{2d-1})$. Hence $\|vq\alpha_{i,d}\|\ll WA^2P^{-d}$
for $v$ from $N=1$, and iteration gives $W\ll V^{2DR}<Y$, a contradiction.

Taking roots and summing the divisor weights gives
$|S_{\boldsymbol F}(\boldsymbol\alpha)|\ll P^{n+\varepsilon}Y^{-E}$.
Summing over the dyadic sets $\mathfrak M(2Y)\setminus\mathfrak M(Y)$,
each of measure $\ll P^{-\mathcal D}Y^{R+1}$, and using the bound with
$Y=P$ on $\mathfrak m(P)$ completes the proof.
\end{proof}

\begin{remark}\label{pre for ext}
    If $\mathfrak{n} \subseteq \mathfrak{m}(Q)$, we still have 
    $$
      \int_{\mathfrak{n}}S_{\boldsymbol{F}}(\boldsymbol{\alpha})d\boldsymbol{\alpha} \ll P^{n-\mathcal{D} }Q^{-\delta}.
    $$
\end{remark}

\section{Gauss sums for systems of forms}  
\subsection{An estimate for Gauss sums} 
During the process of estimating contribution from the major arcs, one encounters Gauss sums associated to our system $\boldsymbol{F}$ against  
Dirichlet characters, for example the sum over $\boldsymbol{h}$ in \eqref{prestep of bounding Mj} below. 
We are going to estimate Gauss sums of this kind on average in the next lemma that is necessary for calculating 
the contribution from the enlarged major arcs.  

From now on we specify 
\begin{equation}\label{def/pi}
\varpi=\frac{1}{4(R+5)}. 
\end{equation}

\begin{lemma}\label{lem/bound/gauss}
Suppose that $\boldsymbol F$ is nonsingular.
Let $\chi_j \bmod k_j $ be primitive characters for $j = 1, \ldots , n$, and put $k_0 = [k_1, \ldots, k_n]$.
Let $\chi^0$ denote the principal character modulo $q$. 
Suppose that 
\begin{equation}\label{Lem81/n>} 
n\ge D^24^{D+6}R^5. 
\end{equation}
Then
\begin{equation}\label{prestep of bounding Mj}
        \sum_{\substack{q \le Q \\ k_0|q}}
        \frac{1}{\varphi^n(q)}\ \sideset{}{^{\dagger}}\sum_{\boldsymbol{a} \bmod q} \bigg| \ \sideset{}{^\ast} 
        \sum_{\boldsymbol{h} \bmod q}\bar{\chi}_1\chi^0(h_1) \cdots \bar{\chi}_n\chi^0(h_n)
        e\bigg(\frac{\boldsymbol{a}\cdot\boldsymbol{F}(\boldsymbol{h})}{q}\bigg)\bigg| 
        \ll k_0^{-\frac{3}{2}+\varepsilon}
\end{equation}
and
\begin{equation}\label{Bq bound}
        \sideset{}{^{\dagger}}\sum_{\boldsymbol{a} \bmod q} \bigg| \ \sideset{}{^\ast}\sum_{\boldsymbol{h} \bmod q}e\bigg(\frac{\boldsymbol{a}\cdot\boldsymbol{F}(\boldsymbol{h})}{q}\bigg)\bigg| 
     \ll q^{n-\frac{3}{2}}.
\end{equation} 
Here
$$
\sideset{}{^{\dagger} }\sum_{\boldsymbol{a} \bmod q}=\sum_{\substack{1\le \boldsymbol{a}\le q\\ (a_1,\ldots,a_R,q)=1}}, 
\quad 
\sideset{}{^\ast}\sum_{\boldsymbol{h} \bmod q}=\sum_{\substack{1\le \boldsymbol{h}\le q\\ (h_i,q)=1}}.
$$
\end{lemma}

Note that the bound \eqref{Lem81/n>} in the above lemma is more restrictive than the condition \eqref{Lem75/n>}  
in Lemma \ref{lemma over minor2} on the minor arcs. Lemma \ref{lem/bound/gauss} needs a stronger assumption on $n$ 
because of the stronger saving $-3/2$ on the right-hand side of \eqref{prestep of bounding Mj} or \eqref{Bq bound}. 
Of course the final condition on $n$ in \eqref{Thm1/n>}  of Theorem \ref{Thm1} comes from \eqref{Lem81/n>}. 

\begin{proof}
    With $b(\boldsymbol h)=\prod_{j=1}^n(\bar\chi_j\chi^0)(h_j)$
it suffices to show that
\[
 \sideset{}{^\dagger}\sum_{\boldsymbol a\bmod q}
 \bigg|\sum_{\boldsymbol h\bmod q}b(\boldsymbol h)
 e\bigg(\frac{\boldsymbol a\cdot\boldsymbol F(\boldsymbol h)}q\bigg)\bigg|
 \ll q^{n-3/2}.
\]

Applying Cauchy's inequality, we have
    \begin{equation}\label{app Cauchy}
        \sideset{}{^{\dagger}}\sum_{\boldsymbol{a} \bmod q} \bigg|\sum_{\boldsymbol{h} \bmod q}b(\boldsymbol{h})e\bigg(\frac{\boldsymbol{a}\cdot\boldsymbol{F}(\boldsymbol{h})}{q}\bigg)\bigg|\le q^{\frac{R}{2}} \bigg( \ \sideset{}{^{\dagger}}\sum_{\boldsymbol{a} \bmod q} \bigg|\sum_{\boldsymbol{h} \bmod q}b(\boldsymbol{h})e\bigg(\frac{\boldsymbol{a}\cdot\boldsymbol{F}(\boldsymbol{h})}{q}\bigg)\bigg|^2 \ \bigg)^{\frac{1}{2}}.
    \end{equation}
    Squaring out the inner term above, we get
    $$
    \bigg|\sum_{\boldsymbol{h} \bmod q}b(\boldsymbol{h})e\bigg(\frac{\boldsymbol{a}\cdot\boldsymbol{F}(\boldsymbol{h})}{q}\bigg)\bigg|^2=\sum_{\boldsymbol{h_1} \bmod q}\sum_{\boldsymbol{h_2} \bmod q}b(\boldsymbol{h_1})\bar{b}(\boldsymbol{h_2})e\bigg(\frac{\boldsymbol{a}\cdot(\boldsymbol{F}(\boldsymbol{h_1})-\boldsymbol{F}(\boldsymbol{h_2}))}{q}\bigg),
    $$
    which is of the form
    \begin{equation*}
        S_{\boldsymbol{F}^{\ast}}\bigg(\frac{\boldsymbol{a}_{\ast}}{q}\bigg)=\sum_{1\le \boldsymbol{k}\le q}\lambda (\boldsymbol{k})e\bigg(\frac{\boldsymbol{a}_{\ast}\cdot\boldsymbol{F}^{\ast}(\boldsymbol{k})}{q}\bigg),
    \end{equation*}
    where $\boldsymbol k=(\boldsymbol h_1,\boldsymbol h_2)$,
$\boldsymbol a_\ast=(\boldsymbol a,-\boldsymbol a)$,
$\lambda(\boldsymbol k)=b(\boldsymbol h_1)\overline{b(\boldsymbol h_2)}$
and $\boldsymbol F^\ast(\boldsymbol k)
=(\boldsymbol F(\boldsymbol h_1),\boldsymbol F(\boldsymbol h_2))$.
Write $\boldsymbol k=\boldsymbol h\boldsymbol j\bmod q$ coordinatewise.
Each unit has $\varphi(q)$ such representations, so multiplicativity
and two applications of Cauchy's inequality give
\[
 \bigg|S_{\boldsymbol F^\ast}\bigg(\frac{\boldsymbol a_\ast}{q}\bigg)\bigg|
 \ll q^\varepsilon U^{1/4},
\]
where
\[
\begin{aligned}
 \boldsymbol L
 =\sum_{\mu,\nu=0}^1(-1)^{\mu+\nu}
   \boldsymbol F^\ast(\boldsymbol h_\mu\boldsymbol j_\nu), \quad
 U=
 \bigg|\sum_{\boldsymbol h_0,\boldsymbol h_1,\boldsymbol j_0,\boldsymbol j_1
              \in(\mathbb Z/q\mathbb Z)^{2n}}
 e\bigg(\frac{\boldsymbol a_\ast\cdot\boldsymbol L}{q}\bigg)\bigg|.
\end{aligned}
\]
The two Jacobian blocks involve disjoint variables, so the rank
estimate in the proof of Lemma \ref{lemma over minor2} with $n_1=0$
gives $8n-B_{2d}(8n,\boldsymbol L)\ge n/2$.

Put $V=(q^\varpi)^{1/(9D(2R+1))}$ and $Z=V^{1/(4R)}$.
Descend in $d$ as in the proof of Lemma~\ref{lemma over minor2},
with $r$ clearing the higher-degree denominators, initially $r=1$.
Apply \cite[Lemma 4.1]{BroHB} with modulus $r$, $\varphi=1$
and $K=q/(rZ)$. A nonzero minor in degree $2d$ gives an integer
$v\ll Z^{(2d-1)2r_d}$ such that $vr$ also clears this degree,
since the polar errors are at most $rZ^{2d-1}/q^{2d}<1/q$.
This cannot occur in every degree, since primitivity would then give
a positive multiple of $q$ of size $\ll Z^{2R(2D-1)}<q$.
Thus in some degree $2d$ all the admissible tuples are rank deficient,
and \cite[Lemmas 4.1 and 5.1]{BroHB} give
\[
 (q^{-8n}U)^{2^{2d-1}}\ll Z^{-n/2}(\log q)^{8n}.
\]
The left-hand side of \eqref{app Cauchy} is therefore
\[
 \ll q^{n+R+\varepsilon}Z^{-n/(8\cdot4^D)}
 =q^{\,n+R+\varepsilon-
       n\varpi/(288\cdot4^D D R(2R+1))}.
\]
By \eqref{Lem81/n>} and \eqref{def/pi}, since $D\ge3$ and $R\ge2$,
\[
 n\ge D^24^{D+6}R^5
 >2940\cdot4^D D^2R^5
 \ge \frac{288\cdot4^D D R(2R+1)(R+\tfrac32)}{\varpi}.
\]
Taking $\varepsilon$ sufficiently small gives the required
bound $q^{n-3/2}$.
\end{proof}

\subsection{The singular series and singular integral}   
The local density of \eqref{equationF=0} at the place $p$ is 
\begin{equation}\label{eq:local-p} 
    \mathfrak{S}_p=\lim_{k \to \infty}\frac{p^{Rk}}{\varphi(p^k)^n}\mathcal{N}(p^k),
\end{equation}
where
$$
\mathcal{N}(q)=\# \{\boldsymbol{x}\in ((\mathbb{Z} /q\mathbb{Z} )^\ast)^n : 
F_{i,d}(\boldsymbol{x})\equiv \boldsymbol{0} \bmod q, \forall i,d  \}
$$
and $\varphi(\cdot) $ is the Euler totient function. We put 
\begin{align*}
C(q,\boldsymbol{a}) 
:=
C_{\boldsymbol{F}} (q,\boldsymbol{a})=
\sideset{}{^\ast}\sum_{\boldsymbol{h} \bmod q}
e\bigg(
\frac{\boldsymbol{a}\cdot \boldsymbol{F}(\boldsymbol{h})}{q}\bigg)
\end{align*}
where $\boldsymbol{h}\in {\mathbb Z}^n$ and $\boldsymbol{a} \in {\mathbb Z}^R$, and write 
\begin{align*} 
B(q)=\sideset{}{^{\dagger}}\sum_{\boldsymbol{a} \bmod q}  C(q,\boldsymbol{a}). 
\end{align*}
Define 
\begin{align}\label{define SH}
\mathfrak{S}_{\boldsymbol{F}}(H)=\sum_{q \le H}\frac{1}{\varphi (q)^n} B(q). 
\end{align}
It follows from Lemma \ref{lem/bound/gauss} with all the characters trivial and $k_0=1$ that, 
if $n\ge D^24^{D+6}R^5$ then as $H\to \infty$ the above $\mathfrak{S}_{\boldsymbol{F}}(H)$ is absolutely convergent 
to $\mathfrak{S}_{\boldsymbol{F}}$, say, and  
\begin{equation}\label{eq:local-inf}
|\mathfrak{S}_{\boldsymbol{F}}(H)-
\mathfrak{S}_{\boldsymbol{F}}|  
\ll H^{-\frac{1}{2}+\varepsilon}. 
\end{equation}
It is worth mentioning that the convergence of the singular series requires far fewer variables, like \eqref{conv/cond}.
For this $\mathfrak{S}_{\boldsymbol{F}}$, we have 
\begin{align}\label{S=prod sigma p}
\mathfrak{S}_{\boldsymbol{F}}=\prod_p\mathfrak{S}_p,
\end{align}
where $\mathfrak{S}_p $ is the local density defined in \eqref{eq:local-p}. 
We remark that $\mathfrak{S}_{\boldsymbol{F}}>0$ if condition (i) of Theorem \ref{Thm1} is satisfied.

We define 
\begin{align*}
\mathfrak{I}_{\boldsymbol{F}}(H)=\int_{|\boldsymbol{\theta}|\le H}\upsilon(\boldsymbol{\theta})d\boldsymbol{\theta},  \quad  \upsilon(\boldsymbol{\theta})=\int_{\mathfrak{B}}e\big(\boldsymbol{\theta}\cdot \boldsymbol{F}(\boldsymbol{x})\big)d\boldsymbol{x}. 
\end{align*}
Recall that \eqref{admissible} for $d=0$ is
\begin{align}\label{conv/cond}
s_{1}(n, \boldsymbol{F})+\sum_{j=1}^D s_j(n, \boldsymbol{F})r_j<1. 
\end{align} 
By \cite[Lemma 8.3]{BroHB} as well as the discussion after it, if $n$ satisfies \eqref{conv/cond} then as $H\to \infty$ the above $\mathfrak{I}_{\boldsymbol{F}}(H)$ is absolutely convergent 
to $\mathfrak{I}_{\boldsymbol{F}}$, say, and  
\begin{equation}\label{Int/Par/Inf}
|\mathfrak{I}_{\boldsymbol{F}}-\mathfrak{I}_{\boldsymbol{F}}(H)|\ll H^{-1}.
\end{equation}
The assumption $n\ge D^24^{D+6}R^5$ implies \eqref{conv/cond}, and consequently yields the inequality \eqref{Int/Par/Inf}. 
More precisely,
\begin{equation}\label{Def/Sin/Int}
    \mathfrak{I}_{\boldsymbol{F}}=\int_{\mathbb R^R} \int_{\mathfrak{B}}e\big(\boldsymbol{\theta}\cdot \boldsymbol{F}(\boldsymbol{x})\big)d\boldsymbol{x} d\boldsymbol{\theta} 
\end{equation}
is the local density of \eqref{equationF=0} at $\infty$.  
We also remark that $\mathfrak{I}_{\boldsymbol{F}}>0$ if $\mathfrak{B}$ contains the real point $x_0$ in (ii) of Theorem \ref{Thm1} in its interior. 

In addition if we insert a continuously differentiable function $\Phi (\boldsymbol{x})$ to the integrand of $\upsilon(\boldsymbol{\theta})$, i.e.,
\begin{align*}
   \int_{\mathfrak{B}}e\big(\boldsymbol{\theta}\cdot \boldsymbol{F}(\boldsymbol{x})\big)\Phi (\boldsymbol{x})d\boldsymbol{x},
\end{align*}
then integrating first in $\boldsymbol\theta$ and using
$\left|\int_{-H}^H e(\theta v)\,d\theta\right|
\ll\min\{H,|v|^{-1}\}$, we obtain 
\begin{align}\label{int with factor}
    \int_{|\boldsymbol{\theta}|\le H}\int_{\mathfrak{B}}e\big(\boldsymbol{\theta}\cdot \boldsymbol{F}(\boldsymbol{x})\big)\Phi (\boldsymbol{x})d\boldsymbol{x}d\boldsymbol{\theta} \ll \|\Phi\|_\infty(\log(2+H))^R
\end{align}
uniformly for $H\ge1$ and bounded measurable $\Phi$.

\section{Contribution from the major arcs and proof of Theorem \ref{Thm2}} 

In \S\S9.1-9.4, we prove the following Lemma \ref{lem/maj/con} for contribution on the major arcs, 
from which we deduce Theorem \ref{Thm2} in \S9.5. 

\begin{lemma}\label{lem/maj/con}
    Let $\boldsymbol{F}=(F_1,\ldots,F_R)=(F_{i,d})_{\substack{d\in \Delta \\ 1\le i\le r_d }}$ be a nonsingular system of forms with each $F_{i,d}$ of degree $d$ in $n$ variables and $\mathfrak{M}=\mathfrak{M}(Q)$ be as in \eqref{define MQ}.
    Let $S_{\boldsymbol{F}}(\boldsymbol{\alpha})$ be as in \eqref{def/SFa}.
    Let $Q=P^{\varpi}$ with \eqref{def/pi} and $$n\ge D^24^{D+6}R^5.$$
    Then 
    \begin{align*}
        \int_{\mathfrak{M}}S_{\boldsymbol{F}}(\boldsymbol{\alpha})d\boldsymbol{\alpha}=
    \mathfrak{S}_{\boldsymbol{F}} \mathfrak{I}_{\boldsymbol{F}}P^{n-\mathcal{D} }+O(P^{n-\mathcal{D}}(\log P)^{-A}),
    \end{align*}
    where $\mathcal{D}$ is as in \eqref{def/R/calD} and $A>0$ is any fixed constant.
\end{lemma}

\subsection{Expression for contribution from the major arcs}  
We start from the single sum over $x\in \mathcal{B}_1(P)$ where $\mathcal{B}_1(P)$ is as in \S4, to get 
\begin{equation}\label{one variable representation}
   \sum_{x\in \mathcal{B}_1(P)}\Lambda(x)e(\boldsymbol{\alpha}\cdot \boldsymbol{F}(x, \ldots))=\sum_{\substack{x\in \mathcal{B}_1(P)\\ (x,q)=1}}\Lambda(x)e(\boldsymbol{\alpha}\cdot \boldsymbol{F}(x, \ldots))+ O((\log qP)^2).
\end{equation}
Since $\boldsymbol{\alpha}=\frac{\boldsymbol{a}}{q}+\boldsymbol{\theta}$, the sum on the right is
\begin{equation*}
    \begin{split}
       & \ \ \sum_{\substack{1\le h\le q \\ (h,q)=1}}
       e\bigg(\sum_{d \in \Delta }\sum_{1\le i\le r_d}\frac{a_{i,d}F_{i,d}(h, \ldots)}{q}\bigg)\sum_{\substack{x\in \mathcal{B}_1(P) \\ x \equiv h\bmod q}}\Lambda(x)e\bigg(\sum_{d \in \Delta }\sum_{1\le i\le r_d}\theta_{i,d}F_{i,d}(x, \ldots)\bigg) \\
       &=\frac{1}{\varphi (q)}\sum_{\chi \bmod q} \, \sideset{}{^\ast}\sum_{h \bmod q}\bar{\chi}(h)
       e\bigg(\frac{\boldsymbol{a}\cdot\boldsymbol{F}(h, \ldots)}{q}\bigg)\sum_{x\in \mathcal{B}_1(P) }\Lambda(x)\chi(x)e(\boldsymbol{\theta }  \cdot\boldsymbol{F}(x, \ldots)).
    \end{split}
\end{equation*}
Here the $\ast$ on the summation over $h$ means that $h$ just runs through the reduced residue classes of $q$.
Then we reduce all the characters modulo $q$ to primitive characters. 
Recalling that if a character $\chi \bmod q$ is induced by a primitive character $\chi^* \bmod k$, then $k \vert q$ and $\chi = \chi^* \chi^0$, where $\chi^0$ is the principal character modulo $q$.
Hence the above is
$$
\frac{1}{\varphi (q)} \sum_{k \vert q}\sideset{}{^\ast}\sum_{\chi \bmod k} \, \sideset{}{^\ast}\sum_{h \bmod q}\bar{\chi}\chi^0(h) e\bigg(\frac{\boldsymbol{a}\cdot\boldsymbol{F}(h, \ldots)}{q}\bigg)\sum_{x\in \mathcal{B}_1(P) }\Lambda(x)\chi\chi^0(x)e(\boldsymbol{\theta }  \cdot\boldsymbol{F}(x, \ldots)),
$$
where the summation over $\chi$ means that $\chi$ goes through primitive characters modulo $k$.
Furthermore the above innermost sum over $x$ can be decomposed as
$$
\sum_{x\in \mathcal{B}_1(P)}\delta^+(\chi\chi^0;q)e(\boldsymbol{\theta }  \cdot\boldsymbol{F}(x, \ldots))+\sum_{x\in \mathcal{B}_1(P)}\delta^-(x, \chi\chi^0;q)e(\boldsymbol{\theta }  \cdot\boldsymbol{F}(x, \ldots)),  
$$
where
$$
\delta^+(\chi\chi^0;q)=
\begin{cases}
   1, &\mbox{$\chi\chi^0 \bmod q$ is principal,} \\
   0, &\mbox{otherwise,}
\end{cases}
$$
and
\begin{equation}\label{define delta-}
    \delta^-(x,\chi\chi^0;q)=\Lambda(x)\chi\chi^0(x)-\delta^+(\chi\chi^0;q).
\end{equation}
From now on, we will abbreviate $\delta^+(\chi\chi^0;q)$ and $\delta^-(x,\chi\chi^0;q)$ by, respectively, $\delta^+$ and $\delta^-$, when their variables are clear.
Inserting these into \eqref{one variable representation}, we see that
\begin{equation}\label{Eq9/3/}
    \begin{split}
        &\sum_{x\in \mathcal{B}_1(P)}\Lambda(x)e(\boldsymbol{\alpha}\cdot \boldsymbol{F}(x, \ldots)) \\
        & =\frac{1}{\varphi (q)} \sum_{k \vert q}\sideset{}{^\ast}\sum_{\chi \bmod k} \, \sideset{}{^\ast}\sum_{h \bmod q}\bar{\chi}\chi^0(h)e\bigg(\frac{\boldsymbol{a}\cdot\boldsymbol{F}(h, \ldots)}{q}\bigg)\sum_{x\in \mathcal{B}_1(P) }(\delta^+ +\delta^-)e(\boldsymbol{\theta }  \cdot\boldsymbol{F}(x, \ldots)) \\ 
        & \ \ +O((\log P)^2). 
    \end{split}
\end{equation}

To simplify our subsequent treatment, we expand each major arc $ \mathfrak{M}(q,\boldsymbol{a};Q)$ slightly to 
$$
\mathfrak{N}(q,\boldsymbol{a};Q)
=\bigg\{(\alpha_{i,d})_{\substack{d \in \Delta \\1\le i\le r_d}}\in \R^R:\ \bigg|\alpha_{i,d}-\frac{a_{i,d}}{q}\bigg|\le \frac{Q}{P^d} \ \textrm{ for all } i,d \bigg\}
$$
and denote by $\mathfrak{N}$ the union of these $\mathfrak{N}(q,\boldsymbol{a};Q)$, similarly to \eqref{define MQ}. 
Therefore, by $\mathfrak{N}\setminus \mathfrak{M} \subset \mathfrak{m}$ and Remark \ref{pre for ext}, we get
\begin{equation*}
    \int_{\mathfrak{N}\setminus \mathfrak{M}}S_{\boldsymbol{F}}(\boldsymbol{\alpha})d\boldsymbol{\alpha} \ll P^{n-\mathcal{D}}Q^{-\delta},
\end{equation*} 
and it follows from $Q=P^{\varpi}$ and $\varpi>0$ that
\begin{equation}\label{M to N}
    \int_{\mathfrak{M}}S_{\boldsymbol{F}}(\boldsymbol{\alpha})d\boldsymbol{\alpha}=\int_{ \mathfrak{N}}S_{\boldsymbol{F}} (\boldsymbol{\alpha})d\boldsymbol{\alpha}+O(P^{n-\mathcal{D}-\delta}).
\end{equation}
We will find that the integral on $\mathfrak{N}$ is easier to calculate, since the singular integral now does not involve $q$ 
in the limits of integration, and therefore can be separated from the singular series.

Applying \eqref{Eq9/3/} repeatedly with $x_1\in \mathcal{B}_1(P), \ldots, x_n\in \mathcal{B}_1(P)$, 
we have 
\begin{equation}\label{separate into M E}
\int_{ \mathfrak{N}}S_{\boldsymbol{F}}(\boldsymbol{\alpha})d\boldsymbol{\alpha}=M+E,
\end{equation}
where $E$ is the contribution from the error term $O((\log P)^2)$ in \eqref{Eq9/3/} for each $x_j\in \mathcal{B}_1(P)$, 
and 
\begin{equation}\label{define major term}
    \begin{split}
        M
        &=\sum_{q \le Q}\frac{1}{\varphi^n(q)}\sideset{}{^{\dagger}}\sum_{\boldsymbol{a} \bmod q}\sum_{k_1 |q}\cdots \sum_{k_n |q} 
        \, \sideset{}{^\ast}\sum_{\chi_1 \bmod k_1}\cdots \sideset{}{^\ast}\sum_{\chi_n \bmod k_n} \, \sideset{}{^\ast} \sum_{\boldsymbol{h} \bmod q}\bar{\chi}_1\chi^0(h_1)\cdots \bar{\chi}_n \chi^0(h_n) \\
        & \ \ \times e\bigg(\frac{\boldsymbol{a}\cdot\boldsymbol{F}(\boldsymbol{h})}{q}\bigg) 
        \sum_{\pm, \ldots, \pm } \int  \sum_{ \boldsymbol{x}\in P\mathfrak{B} }\delta_1^{\pm } \cdots \delta_n^{\pm }e(\boldsymbol{\theta }  \cdot\boldsymbol{F}(\boldsymbol{x})) d\boldsymbol{\theta}.
    \end{split}
\end{equation}
Here the integration interval for each coordinate $\theta_{i,d}$ in $\boldsymbol{\theta}$ is $|\theta_{i,d}| \le \frac{Q}{P^d}$ 
where $d \in \Delta$ and $1 \le i \le r_d$, and also we have used the abbreviations that 
$$
\delta^+_j=\delta^+_j(\chi_j\chi^0;q), \ \ \delta^-_j=\delta^-_j(x_j, \chi_j\chi^0;q)
$$
for $j=1, \ldots, n$. 

First, $E$ is easy to estimate. In fact 
\begin{align*}
E 
&\ll P^{n-1}(\log P)^c |\mathfrak{N}| \ll P^{n-1}Q^{R+1}\prod_{d \in \Delta}\bigg(\frac{Q}{P^d}\bigg)^{r_d}(\log P)^c \\
&\ll P^{n-\mathcal{D}-1}Q^{2R+1}(\log P)^c.
\end{align*} 
It follows from $Q=P^{\varpi}$ with \eqref{def/pi} that 
\begin{equation}\label{bound E}
E \ll P^{n-\mathcal{D}-\delta}, 
\end{equation}
which is clearly acceptable. 

Next, we turn to $M$. For $j=0, \ldots, n$, denote by $M_j$ the contribution of the product 
$\delta^{\pm }_1 \cdots \delta^{\pm }_n$ where exactly $j$ minus signs occur, so 
that \eqref{define major term} becomes $M = M_0 + M_1 + \cdots + M_n $ with
\begin{equation}\label{separate M into Mj}
    \begin{split}
        M_j
        &=\sum_{k_1 \le Q}\cdots \sum_{k_n \le Q} \, \sideset{}{^\ast}\sum_{\chi_1 \bmod k_1}\cdots \sideset{}{^\ast}\sum_{\chi_n \bmod k_n} \, \sum_{\substack{q \le Q \\ k_0|q}}\frac{1}{\varphi^n(q)}\sideset{}{^{\dagger}}\sum_{\boldsymbol{a} \bmod q} \, \sideset{}{^\ast}\sum_{\boldsymbol{h} \bmod q}\bar{\chi}_1\chi^0(h_1)\cdots \bar{\chi}_n \chi^0(h_n) \\ 
        & \times e\bigg(\frac{\boldsymbol{a}\cdot\boldsymbol{F}(\boldsymbol{h})}{q}\bigg) \sum_{[j-] } \int  \sum_{ \boldsymbol{x}\in P\mathfrak{B} }\delta_1^{\pm } \cdots \delta_n^{\pm }e(\boldsymbol{\theta }  \cdot\boldsymbol{F}(\boldsymbol{x})) d\boldsymbol{\theta} 
    \end{split}
\end{equation}
for $ j = 0, 1, \dots , n$, where $[j -]$ goes through all subsets of $\{\pm , \ldots , \pm\}$ with exactly $j$ minus signs and $k_0=[k_1, \ldots   , k_n]$.
Therefore, by \eqref{M to N}, \eqref{separate into M E}, \eqref{bound E} and \eqref{separate M into Mj}, we have 
\begin{equation}\label{separate the integral on the major arcs}
    \int_{\mathfrak{M}}S_{\boldsymbol{F}}(\boldsymbol{\alpha})d\boldsymbol{\alpha}=M_0+M_1+ \cdots +M_n+O(P^{n-\mathcal{D}-\delta}). 
\end{equation}

\subsection{Estimation of $M_0$}  
Furthermore, we single out $M_0$ which comes from the product $\delta^+_1 \cdots \delta^+_n $, and we will see that 
$M_0$ gives the main term. 
After this, we will show that other terms $M_1, \ldots , M_n$ are negligible.

\begin{lemma}\label{lem/comp/M0}
    Assume the conditions of Lemma~\ref{lem/maj/con}.
    Then
    \begin{equation*}
        M_0=\mathfrak{S}_{\boldsymbol{F}} \mathfrak{I}_{\boldsymbol{F}} P^{n-\mathcal{D}}+O(P^{n-\mathcal{D}-\delta}),
    \end{equation*}
    where $\delta>0$ is a fixed constant.
\end{lemma}

\begin{proof}
By \eqref{separate M into Mj}, $M_0$ denotes the contribution from the term $\delta^+_1 \cdots \delta^+_n $.
Recall that  $\delta^+_j=\delta^+_j(\chi_j\chi^0;q)$, which is equal to $1$ if $\chi_j  \bmod k_j$ is the principal 
character modulo $1$, and equal to $0$ otherwise. We have
    \begin{equation*}
        \begin{split}
            M_0
            &=\sum_{\substack{q \le Q }}\frac{1}{\varphi^n(q)}\sideset{}{^{\dagger}}\sum_{\boldsymbol{a} \bmod q} \, \sideset{}{^\ast}\sum_{\boldsymbol{h} \bmod q}e\bigg(\frac{\boldsymbol{a}\cdot\boldsymbol{F}(\boldsymbol{h})}{q}\bigg)  \int  \sum_{ \boldsymbol{x}\in P\mathfrak{B} }e(\boldsymbol{\theta }  \cdot\boldsymbol{F}(\boldsymbol{x})) d\boldsymbol{\theta} \\
            &=\sum_{\substack{q \le Q }}\frac{1}{\varphi^n(q)}\sideset{}{^{\dagger}}\sum_{\boldsymbol{a} \bmod q} \, \sideset{}{^\ast}\sum_{\boldsymbol{h} \bmod q}e\bigg(\frac{\boldsymbol{a}\cdot\boldsymbol{F}(\boldsymbol{h})}{q}\bigg)  \int  \int_{ P\mathfrak{B} }e(\boldsymbol{\theta }  \cdot\boldsymbol{F}(\boldsymbol{x})) d\boldsymbol{x} d\boldsymbol{\theta} \\
            & \ \ +O\bigg(\sum_{\substack{q \le Q }}\frac{1}{\varphi^n(q)}q^Rq^n\prod_{d \in \Delta}\bigg(\frac{Q}{P^d}\bigg)^{r_d}QP^{n-1}\bigg). 
        \end{split}
    \end{equation*}
Here the partial singular series and the partial singular integral are already separated, and therefore,  
by \eqref{define SH} and \eqref{Def/Sin/Int}, we may write
    $$
    M_0=\mathfrak{S}_{\boldsymbol{F}}(Q) \mathfrak{I}_{\boldsymbol{F}}(Q)P^{n-\mathcal{D}}+O(P^{n-\mathcal{D}-1}Q^{2R+2+\varepsilon}).
    $$
Note that $n\ge D^24^{D+6}R^5$, \eqref{def A1n} and \eqref{t_0 lower bound} guarantee $t_0(n,\boldsymbol F)>0$, 
and it follows that \eqref{conv/cond} holds.
By  \eqref{eq:local-inf} and \eqref{Int/Par/Inf}, we have
    \begin{equation*}
        \begin{split}
           M_0
           &=(\mathfrak{S}_{\boldsymbol{F}}+O(Q^{-\frac{1}{2}+\varepsilon})) (\mathfrak{I}_{\boldsymbol{F}}+O(Q^{-1})) P^{n-\mathcal{D}}+O(P^{n-\mathcal{D}-1}Q^{2R+2+\varepsilon}) \\
           &=\mathfrak{S}_{\boldsymbol{F}} \mathfrak{I}_{\boldsymbol{F}}P^{n-\mathcal{D}}+O(P^{n-\mathcal{D}}Q^{-\frac{1}{2}+\varepsilon})+O(P^{n-\mathcal{D}-1}Q^{2R+2+\varepsilon}).
        \end{split}
\end{equation*}
The choice $Q=P^{\varpi}$ with \eqref{def/pi} completes the proof.
\end{proof}

\subsection{Estimation of $M_n$}  
We now come to bound $M_1, \ldots, M_n$. Among all $M_1, \ldots , M_n$, the most complicated one is $M_n$. 
We treat $M_n$ in full detail in the following, and then indicate how to modify this treatment to control $M_1, \ldots , M_{n-1}$ in the 
next subsection. 

By \eqref{separate M into Mj} with $j=n$, we have
\begin{equation}\label{write Mn}
    \begin{split}
        M_n
        &= \sum_{k_1 \le Q}\cdots \sum_{k_n \le Q} \, \sideset{}{^\ast}\sum_{\chi_1 \bmod k_1}\cdots \sideset{}{^\ast}\sum_{\chi_n \bmod k_n} \, \sum_{\substack{q \le Q \\ k_0|q}}\frac{1}{\varphi^n(q)}\sideset{}{^{\dagger}}\sum_{\boldsymbol{a} \bmod q} \, \sideset{}{^\ast}\sum_{\boldsymbol{h} \bmod q}\bar{\chi}_1\chi^0(h_1)\cdots \bar{\chi}_n\chi^0(h_n) \\ 
        & \times e\bigg(\frac{\boldsymbol{a}\cdot\boldsymbol{F}(\boldsymbol{h})}{q}\bigg) \int  \sum_{ \boldsymbol{x}\in P\mathfrak{B} }\delta_1^- \cdots \delta_n^-e(\boldsymbol{\theta }  \cdot\boldsymbol{F}(\boldsymbol{x})) d\boldsymbol{\theta},
    \end{split}
\end{equation}
where $\delta^-_s=\delta^-(x_s, \chi_s\chi^0;q)$ for $s=1, \ldots, n$, $\chi^0$ is the principal character modulo $q$ 
and $k_0=[k_1, \dots, k_n]$. 
Then 
\begin{align}\label{def/Mn}
|M_n|
&\leq \sum_{k_1 \leq Q}\cdots \sum_{k_n \leq Q} \, \sideset{}{^\ast}\sum_{\chi_1 \bmod  {k_1}}\cdots \sideset{}{^\ast}\sum_{\chi_n \bmod  {k_n}} \, \sum_{\substack{q \leq Q \\ k_0 \mid q}}\frac{1}{\varphi^n(q)}
|\nu (q; \chi_1, \ldots, \chi_n)| \nonumber\\
&\quad \times |\mathfrak{I}(q; \chi_1, \ldots, \chi_n)|,  
\end{align}
where 
\begin{align*}
\nu (q; \chi_1, \ldots, \chi_n)
&=
\sideset{}{^\dagger}\sum_{\boldsymbol{a} \bmod  q}
 \ \sideset{}{^\ast}\sum_{\boldsymbol{h} \bmod  q}\bar{\chi}_1\chi^0(h_1)\cdots \bar{\chi}_n \chi^0(h_n)
 e\bigg(\frac{\boldsymbol{a} \cdot \boldsymbol{F}(\boldsymbol{h})}{q}\bigg) \nonumber\\
&=
\sideset{}{^\dagger}\sum_{\boldsymbol{a} \bmod  q} C_F (q, \boldsymbol{a}; \bar{\chi}_1\chi^0, \ldots, \bar{\chi}_n\chi^0) 
\end{align*}
and
\begin{equation*}
\mathfrak{I}(q; \chi_1, \ldots, \chi_n)=\int \Delta(q, \boldsymbol{\theta}; \chi_1, \ldots, \chi_n) d\boldsymbol{\theta}
\end{equation*}
with 
\begin{equation}\label{def/Ichichi}
\Delta(q, \boldsymbol{\theta}; \chi_1, \ldots, \chi_n)
= \sum_{ \boldsymbol{x}\in P\mathfrak{B} }
\delta^-(x_1, \chi_1\chi^0; q)\cdots \delta^-(x_n, \chi_n\chi^0; q) e(\boldsymbol{\theta} \cdot \boldsymbol{F}(\boldsymbol{x})).
\end{equation}
Note that in \eqref{def/Mn} the integral $\mathfrak{I}(q; \chi_1, \ldots, \chi_n)$ depends on \( q \). 
Therefore, it cannot be taken outside the summation over $q$ for direct application of Lemma \ref{lem/bound/gauss}.

The above two formulae \eqref{def/Mn} and \eqref{def/Ichichi} are the starting points of our estimation of $M_n$. 
Our treatment of $M_n$ falls naturally into two cases:
\begin{enumerate}[(i)]
    \item at least one of these $k_1, \dots, k_n$ is large; 
    \item all of these $k_1, \dots, k_n$ are small.
\end{enumerate}
In each case \(M_n\) denotes the corresponding part of \eqref{write Mn}.
Now we treat the first case.

\begin{lemma}
    Let $M_n$ be as in \eqref{write Mn}.
    Assume the conditions of Lemma~\ref{lem/maj/con}.
    For any constant $A > 0$, there is a constant $B = B(A) > 0$,
    such that if one of the moduli $k_1, \dots , k_n$ is larger than $(\log P)^{4B}$,
    then
    \begin{equation*}
        M_n \ll P^{n-\mathcal{D}}(\log P)^{-A}.
    \end{equation*}
\end{lemma}

\begin{proof}
In \eqref{def/Ichichi}, we employ partial summation formula to each variable $x_j$ separately, getting that it equals to
\begin{equation}\label{rewrite Rn}
\int_{P\mathfrak{B}}e(\boldsymbol{\theta }\cdot\boldsymbol{F}(\boldsymbol{x}))
d\bigg\{ \sum_{m\le x_1}\delta^-(m, \chi_1\chi^0;q)\bigg\}\cdots d\bigg\{ \sum_{m\le x_n}\delta^-(m, \chi_n\chi^0;q)\bigg\}.
\end{equation} 
To analyze \eqref{rewrite Rn}, let $\chi \bmod k$ be a primitive character, $k \vert q$ and $1 \le x \le P$. 
Recall \eqref{define delta-} and the explicit formula (see e.g. \cite[Proposition 5.25]{IwaKow})
    $$
    \sum_{m\le x}\delta^-(m, \chi\chi^0;q)=-\sum_{|\gamma_{\chi}| \le T }\frac{x^{\rho_{\chi} }-1}{\rho_{\chi} }+O\bigg(\frac{P \log P}{T}+(\log P)^2 \bigg),
    $$
where $\rho_{\chi} = \beta_{\chi} + i\gamma_{\chi}$ runs through the non-trivial zeros of the Dirichlet $L$-function $L(s, \chi)$ with $|\gamma_{\chi}| \le T$.
Choosing 
$$
T=P^{\frac{1}{3}}, 
$$
we have
\begin{equation*}
\sum_{m\le x}\delta^-(m, \chi\chi^0;q)=-\sum_{|\gamma_{\chi}| \le T }\frac{x^{\rho_{\chi} }-1}{\rho_{\chi} }+ \sigma(x), \quad  
\sigma(x) \ll P^{\frac{2}{3}} \log P,
\end{equation*}
    and hence 
    \begin{equation*}
        d\bigg\{ \sum_{m\le x}\delta^-(m, \chi\chi^0;q)\bigg\}=-\sum_{|\gamma_{\chi}| \le T }x^{\rho_{\chi}-1}dx+ d\sigma(x).
    \end{equation*}

It is obvious that the above formula still holds when $x$ and $\chi$ are replaced, respectively,  
by $x_j$ and $\chi_j$ with $j=1, \ldots, n$. And it follows that, for $j=1, \ldots, n$,
    \begin{equation}\label{inserting explicit formula}
        \begin{split}
            &\ \ \ \int^{b''_jP}_{b'_jP}e(\boldsymbol{\theta }\cdot\boldsymbol{F}(\boldsymbol{x}))
            d\bigg\{ \sum_{m\le x_j}\delta^-(m, \chi_j\chi^0;q)\bigg\} \\
            &=-\sum_{|\gamma_{\chi_j}|\le T}\int^{b''_jP}_{b'_jP}e(\boldsymbol{\theta }\cdot\boldsymbol{F}(\boldsymbol{x}))x^{\rho_{\chi_j}-1}_j dx_j+\int^{b''_jP}_{b'_jP}e(\boldsymbol{\theta }\cdot\boldsymbol{F}(\boldsymbol{x}))d\sigma(x_j).
        \end{split}
    \end{equation}
    The second part above can be well controlled as follows:
    \begin{equation*}
        \begin{split}
           &\ll |\sigma(b''_jP)|+ |\sigma(b'_jP)|+\bigg|\int^{b''_jP}_{b'_jP}e(\boldsymbol{\theta }\cdot\boldsymbol{F}(\boldsymbol{x}))\frac{\partial(\boldsymbol{\theta }\cdot\boldsymbol{F}(\boldsymbol{x}))}{\partial x_j}\sigma(x_j)dx_j \bigg|\\
           &\ll P^{\frac{2}{3}} \log P \bigg(1+\sum_{d \in \Delta}\frac{Q}{P^d}P^d\bigg) 
           \ll P^{\frac{2}{3}+\varepsilon}Q.
        \end{split}
    \end{equation*} 
Hence \eqref{inserting explicit formula} becomes 
\begin{equation}\label{ins/exp/for/ult}
             \int^{b''_jP}_{b'_jP}e(\boldsymbol{\theta }\cdot\boldsymbol{F}(\boldsymbol{x}))
             d\bigg\{ \sum_{m\le x_j}\delta^-(m, \chi_j\chi^0;q)\bigg\} 
             \mkern-4.5mu=\mkern-4.5mu-\sum_{|\gamma_{\chi_j}|\le T}\int^{b''_jP}_{b'_jP}e(\boldsymbol{\theta }\cdot\boldsymbol{F}(\boldsymbol{x}))x^{\rho_{\chi_j}-1}_j dx_j+O(P^{\frac{2}{3}+\varepsilon}Q).
\end{equation}
    For $j=1,\ldots,n$, we insert \eqref{ins/exp/for/ult} into \eqref{rewrite Rn} to get
    \begin{equation}\label{separate Rn}
        \Delta(q, \boldsymbol{\theta}; \chi_1, \ldots, \chi_n)=E_1+E_2,
    \end{equation}
where $E_1$ denotes the contribution from the product of the main terms in \eqref{ins/exp/for/ult} for all $j = 1, \ldots , n$,  
and $E_2$ the contribution from the error term in \eqref{ins/exp/for/ult} for some $j$.
    More precisely,
    \begin{equation}\label{write E1}
        E_1=\sum_{|\gamma_{1}|\le T} \cdots \sum_{|\gamma_{n}| \le T} (-1)^n\int_{\boldsymbol{x} \in P\mathfrak{B}}e(\boldsymbol{\theta }\cdot\boldsymbol{F}(\boldsymbol{x}))x^{\rho_{1}-1}_1\cdots x^{\rho_{n}-1}_n d\boldsymbol{x},
    \end{equation}
    where we abbreviate $\gamma_{\chi_j}$ and $\rho_{\chi_j}$, respectively, as $\gamma_j$ and $\rho_j$ for all $j=1,\ldots, n$.
    And 
    \begin{equation}\label{bound E2}
        E_2 \ll \sum_{J \subset \{1,\dots, n\}}(P^{\frac{2}{3}+\varepsilon}Q)^{n-|J|}\prod_{j \in J}\bigg\{\sum_{|\gamma_j|\le T}\int^{b''_jP}_{b'_jP}x^{\beta_{j}-1}_j dx_j\bigg\}, 
    \end{equation}
    where $J$ runs through all proper subsets of $\{1, \ldots, n\}$ and $\beta_j$ is the real part of $\rho_j$ for each $j=1, \ldots, n$.
    Therefore it follows from \eqref{def/Mn} and \eqref{separate Rn} that
    \begin{equation}\label{separate Mn}
    \begin{split}
        M_n
        &\ll \sum_{k_1 \leq Q}\cdots \sum_{k_n \leq Q} \, \sideset{}{^\ast}\sum_{\chi_1 \bmod  {k_1}}\cdots \sideset{}{^\ast}\sum_{\chi_n \bmod  {k_n}} \ \sum_{\substack{q \leq Q \\ k_0 \mid q}}\frac{1}{\varphi^n(q)}
        |\nu (q; \chi_1, \ldots, \chi_n)| \\
        &\quad \times \bigg\{\bigg| \int E_1 d \boldsymbol{\theta}\bigg|+\int|E_2| d\boldsymbol{\theta}\bigg\} \\
        &=:M^{(M)}_n+M^{(E)}_n.
    \end{split}
    \end{equation}
    In $M^{(M)}_n$, it is crucial that the absolute value symbol is outside the integral.   
    
    Now we begin to handle $M^{(M)}_n$ and $M^{(E)}_n$ separately. 
    First we will give some estimates for $M^{(M)}_n$.
    Changing variables $\boldsymbol{x}\rightarrow P\boldsymbol{x}$, $\theta_{i,d}\rightarrow P^{-d}\theta_{i,d}$ for each $d \in \Delta$ and $1 \le i \le r_d$, 
    we get
    \begin{equation}\label{bound int E1}
        \int E_1 d \boldsymbol{\theta}=(-1)^n\sum_{|\gamma_{1}| \le T} \cdots \sum_{|\gamma_{n}| \le T} P^{\rho_1+\cdots+\rho_n-\mathcal{D}}\int_{|\boldsymbol{\theta}|\le Q }\int_{\boldsymbol{x} \in \mathfrak{B}}e(\boldsymbol{\theta }\cdot\boldsymbol{F}(\boldsymbol{x}))\Phi (\boldsymbol{x}) d\boldsymbol{x}d\boldsymbol{\theta}
    \end{equation}
    with
    $$
    \Phi (\boldsymbol{x})=x^{\rho_{1}-1}_1\cdots x^{\rho_{n}-1}_n,
    $$
    where we recall that $\mathcal{D}$ is as in \eqref{def/R/calD}.
    By \eqref{int with factor} and the uniform bound
$\|\Phi\|_\infty\ll1$, the double integral on the right-hand
side of \eqref{bound int E1} is $O((\log P)^R)$,
uniformly in the zeros.
    Consequently
    $$
    \bigg|\int E_1 d \boldsymbol{\theta}\bigg|\ll (\log P)^R \sum_{|\gamma_{1}| \le T} \cdots \sum_{|\gamma_{n}|\le T} P^{\beta_1+\cdots+\beta_n-\mathcal{D}},
    $$
    and thus
\begin{equation}\label{bound MnM 1}
    \begin{split}
            M^{(M)}_n
            &=\sum_{k_1 \le Q}\cdots \sum_{k_n \le Q} \, \sideset{}{^\ast}\sum_{\chi_1 \bmod {k_1}}\cdots \sideset{}{^\ast}\sum_{\chi_n \bmod {k_n}}  \sum_{\substack{q \leq Q \\ k_0 \mid q}}\frac{1}{\varphi^n(q)}
|\nu (q; \chi_1, \ldots, \chi_n)| \bigg| \int E_1 d \boldsymbol{\theta}\bigg| \\
            &\ll (\log P)^R \sum_{k_1 \le Q}\cdots \sum_{k_n \le Q} \, \sideset{}{^\ast}\sum_{\chi_1 \bmod {k_1}}\cdots \sideset{}{^\ast}\sum_{\chi_n \bmod {k_n}} \sum_{|\gamma_{1}| \le T} \cdots \sum_{|\gamma_{n}| \le T} P^{\beta_1+\cdots+\beta_n-\mathcal{D}} \\
            & \quad \times \sum_{\substack{q \leq Q \\ k_0 \mid q}}\frac{1}{\varphi^n(q)} |\nu (q; \chi_1, \ldots, \chi_n)| \\
            &\ll (\log P)^R \sum_{k_1 \le Q}\cdots \sum_{k_n \le Q} \, k^{-\frac{3}{2}+\varepsilon}_0 \, \sideset{}{^\ast}\sum_{\chi_1 \bmod {k_1}}\cdots \sideset{}{^\ast}\sum_{\chi_n \bmod {k_n}} \sum_{|\gamma_{1}| \le T} \cdots \sum_{|\gamma_{n}| \le T} P^{\beta_1+\cdots+\beta_n-\mathcal{D}},
    \end{split}
\end{equation}
where the last inequality follows from Lemma \ref{lem/bound/gauss}.
    Since there is a $j$ such that $k_j \ge (\log P)^{4B}$, we must have $k_0 \ge (\log P)^{4B}$, and then the above is
    \begin{equation}\label{bound MnM 2}
        M^{(M)}_n \ll P^{n-\mathcal{D}}(\log P)^{-5B+R}
        \bigg\{ \sum_{k\le Q} \, \sideset{}{^\ast}\sum_{\chi \bmod k}\sum_{|\gamma|\le T}P^{\beta-1} \bigg\}^n. 
    \end{equation}
    This saving of powers of $\log P$ is crucial in our argument.

It is already in a good shape, and the quantity within the brackets can be dealt with by the zero-density estimate of the large-sieve type 
(see e.g. \cite[Theorem 1]{Mon}) that
\begin{equation}\label{zero desity estimate}
        \sum_{k\le Q} \, \sideset{}{^\ast}\sum_{\chi \bmod k} \, 
        \sum_{\substack{\sigma \le \beta \le 1 \\ |\gamma| \le T}}1 \ll (Q^2T)^{\frac{12}{5}(1-\sigma)}(\log QT)^{13}, 
\end{equation}
where $\beta + i\gamma$ runs through non-trivial zeros of the Dirichlet $L$-function $L(s, \chi)$ with $\sigma \le \beta \le 1$ and $|\gamma| \le T$.
Employing \eqref{zero desity estimate} and partial summation formula, we get 
    \begin{equation}\label{using zero density}
        \begin{split}
            \sum_{k\le Q} \, \sideset{}{^\ast}\sum_{\chi \bmod k}\sum_{|\gamma| \le T}P^{\beta-1} 
            &\ll (\log P)^{13} \int^1_{\frac{1}{2}}P^{\sigma-1}(Q^2T)^{\frac{12}{5}(1-\sigma)}d\sigma \\
            &\ll (\log P)^{13} \max_{\frac{1}{2}\le \sigma \le 1}\bigg( \frac{Q^{\frac{24}{5}}T^{\frac{12}{5}}}{P}\bigg)^{1-\sigma} \ll (\log P)^{13},
        \end{split}
    \end{equation}
    where we have applied $T=P^{\frac{1}{3}}$, $Q=P^{\varpi}$ with \eqref{def/pi} and $R\ge 2$.
 We now remark that the choice of $T$ is also an embody for art of balancing. 
 Only when $\varpi$ is chosen sufficiently small, does the suitable choice of $T$ exist. 

    Therefore,  \eqref{bound MnM 2} becomes
    \begin{equation}\label{MnM bound}
        M^{(M)}_n \ll P^{n-\mathcal{D}}(\log P)^{-A}
    \end{equation}
    with $5B=A+R+13n$.

    Next we turn to the estimation of $M^{(E)}_n$, for which a rough estimate suffices. We start from \eqref{separate Mn}, getting
\begin{equation}\label{estimating MnE}
    \begin{split}
            M^{(E)}_n
            &=\sum_{k_1 \le Q}\cdots \sum_{k_n \le Q} \, \sideset{}{^\ast}\sum_{\chi_1 \bmod {k_1}}\cdots \sideset{}{^\ast}\sum_{\chi_n \bmod {k_n}} \sum_{\substack{q \leq Q \\ k_0 \mid q}}\frac{1}{\varphi^n(q)} |\nu (q; \chi_1, \ldots, \chi_n)| \int| E_2| d\boldsymbol{\theta} \\
            &\ll \sum_{k_1 \le Q}\cdots \sum_{k_n \le Q} \, k^{-\frac{3}{2}+\varepsilon}_0 \, \sideset{}{^\ast}\sum_{\chi_1 \bmod {k_1}}\cdots \sideset{}{^\ast}\sum_{\chi_n \bmod {k_n}}  \sum_{J \subset \{1,\dots, n\}}(P^{\frac{2}{3}+\varepsilon}Q)^{n-|J| }  \\
            & \quad \times \prod_{j \in J}\bigg\{\sum_{|\gamma_j|\le T}\int^{b''_jP}_{b'_jP}x^{\beta_{j}-1}_j dx_j\bigg\} 
            \prod_{d \in \Delta}\bigg(\frac{Q}{P^d}\bigg)^{r_d}     \\
            &\ll P^{n-\mathcal{D}+\varepsilon}Q^{R}\sum_{J \subset \{1,\dots, n\}} (P^{-\frac{1}{3}}Q^3)^{n-| J| } 
            \bigg\{\sum_{k \le Q} \, \sideset{}{^\ast}\sum_{\chi \bmod k}\sum_{|\gamma|\le T}P^{\beta-1}\bigg\}^{|J|},
    \end{split}
\end{equation}
where the first inequality follows from Lemma \ref{lem/bound/gauss} and \eqref{bound E2}. 
By \eqref{using zero density} and \eqref{def/pi},
\begin{equation}\label{MnE bound}
        M^{(E)}_n \ll P^{n-\mathcal{D}-\delta},
\end{equation}
where $\delta>0$ is a fixed constant.
Finally, \eqref{separate Mn}, \eqref{MnM bound} and \eqref{MnE bound} yield the desired result. 
\end{proof}

Now, we treat the second case that all the moduli $k_1, \ldots, k_n$ are small.
\begin{lemma}
    Let $M_n$ be as in \eqref{write Mn}.
    Assume the conditions of Lemma~\ref{lem/maj/con}.
    Let $A > 0$ and $B>0$ be both arbitrary.
    If all of the moduli $k_1, \dots , k_n \le (\log P)^B$,
    then
    \begin{equation*}
        M_n \ll P^{n-\mathcal{D}}(\log P)^{-A}.
    \end{equation*}
\end{lemma}
\begin{proof}
    Let $E_1$ and $E_2$ be as in \eqref{write E1} and \eqref{bound E2} respectively.
    We have, similarly to \eqref{separate Mn},
\begin{equation}\label{separate Mn '}
        \begin{split}
            M_n \ll 
            &\sum_{k_1 \le (\log P)^B}\cdots \sum_{k_n \le (\log P)^B} \,  \sideset{}{^\ast}\sum_{\chi_1 \bmod {k_1}}\cdots \sideset{}{^\ast}\sum_{\chi_n \bmod {k_n}} \sum_{\substack{q \leq Q \\ k_0 \mid q}}\frac{1}{\varphi^n(q)} |\nu (q; \chi_1, \ldots, \chi_n)| \\
            &\quad \times \bigg\{\bigg| \int E_1 d \boldsymbol{\theta}\bigg|+\int|E_2| d\boldsymbol{\theta}\bigg\}  \\
            &=:M^{(M)}_n+M^{(E)}_n.
        \end{split}
    \end{equation}

Similar to \eqref{bound MnM 1}, we get
    \begin{equation*}
        \begin{split}
            M^{(M)}_n
            &=\sum_{k_1 \le (\log P)^B}\cdots \sum_{k_n \le (\log P)^B} \, \sideset{}{^\ast}\sum_{\chi_1 \bmod {k_1}}\cdots \sideset{}{^\ast}\sum_{\chi_n \bmod {k_n}}  \sum_{\substack{q \leq Q \\ k_0 \mid q}}\frac{1}{\varphi^n(q)} |\nu (q; \chi_1, \ldots, \chi_n)| \\ 
            &\quad \times \bigg| \int E_1 d \boldsymbol{\theta}\bigg| \\
            &\ll (\log P)^R\sum_{k_1 \le (\log P)^B}\cdots \sum_{k_n \le (\log P)^B} \, \sideset{}{^\ast}\sum_{\chi_1 \bmod {k_1}}\cdots \sideset{}{^\ast}\sum_{\chi_n \bmod {k_n}} \sum_{|\gamma_{1}|\le T} \cdots \sum_{|\gamma_{n}|\le T} P^{\beta_1+\cdots+\beta_n-\mathcal{D}}. 
        \end{split}
    \end{equation*}
Here we cannot expect any saving from the negative power of $k_0$ since at the present situation $k_0$ is small, and therefore we just drop it.
However, the key saving comes from the zero-free region of Chudakov-type.
We have
\begin{equation}\label{MnM ' bound}
        M^{(M)}_n \ll P^{n-\mathcal{D}} (\log P)^R \bigg\{ \sum_{k\le (\log P)^B} \, \sideset{}{^\ast}
        \sum_{\chi \bmod k}\sum_{|\gamma|\le T}P^{\beta-1} \bigg\}^n. 
\end{equation}
    The quantity within the braces will be bounded by the classical zero-density estimates that
    \begin{equation}\label{classsical zero-density estimate}
        \sum_{\chi \bmod k}\sum_{\substack{\sigma \le \beta \le 1 \\ |\gamma| \le T }} 1 \ll (kT)^{\frac{12}{5}(1-\sigma)}(\log kT)^{13}, 
    \end{equation}
    where $\beta+i\gamma$ runs through non-trivial zeros of $L(s,\chi)$ with $\sigma \le \beta \le 1$ and $|\gamma| \le T$ and $\chi$ runs through characters $\bmod \ k$.
    This is not a large-sieve type estimate compared to \eqref{zero desity estimate}.
    Additionally we also need Chudakov's zero-free region for Dirichlet $L$-functions 
    (see e.g. Prachar \cite[Satz VIII.6.2]{Prachar}) that, for any $\chi \bmod  k$, there exists a constant $c_1>0$ such that $L(\sigma+it,\chi)\neq 0$ in the region
    \begin{equation*}
        \sigma \ge 1-\frac{c_1}{\log k+(\log (|t|+2))^{\frac{4}{5}}}
    \end{equation*}
    except for the possible Siegel zero.
    However since we have $k\le (\log P)^B$, the Siegel zero does not exist in the present setting.
    It follows that $L(s,\chi)$ is zero-free for $\sigma \ge 1-\eta(T)$ and $| t| \le T $, where
$$
    \eta(\tau)=\frac{c_1}{2 (\log(|\tau|+2))^{\frac{4}{5}}}
$$
    and $c_1>0$.
    Hence, by \eqref{classsical zero-density estimate},
    \begin{equation}\label{use classical zero-density}
        \begin{split}
            \sum_{\chi \bmod k}\sum_{\substack{ |\gamma| \le T }} P^{\beta-1} 
            &\ll -\int_{\frac{1}{2}}^{1-\eta{T}} P^{\sigma-1}d  
            \bigg\{\sum_{\chi \bmod k}\sum_{\substack{\sigma \le \beta \le 1 \\ |\gamma|\le T}} 1 \bigg\}  \\
            &\ll (\log P)^{13} \max_{\frac{1}{2}\le \sigma \leq 1-\eta(T)}\bigg(\frac{(kT)^{\frac{12}{5}}}{P}\bigg)^{1-\sigma}. 
        \end{split}
    \end{equation}
    Since $k\le (\log P)^B$ and $T=P^{\frac{1}{3}}$, the above maximum is
    \begin{equation}\label{use Chudakov zero-free region}
        \ll (\log P)^{13}\max_{\frac{1}{2}\le \sigma \le 1-\eta(T)} P^{\frac{1}{5}(\sigma-1)} 
        \ll \exp\bigg(-\frac{c_1}{10}\frac{\log P}{ (\log T)^{\frac{4}{5}}}\bigg) \ll \exp(-c_2 (\log P)^{\frac{1}{5}})
    \end{equation}
    for some suitable constant $c_2>0$.
    Inserting \eqref{use classical zero-density} and \eqref{use Chudakov zero-free region} into \eqref{MnM ' bound}, we get 
\begin{equation*}
        M_n^{(M)} \ll P^{n-\mathcal{D}}\exp(-c_3(\log P)^{\frac{1}{5}}),
\end{equation*}
    where $c_3>0$ is a fixed constant.

    Now we turn to $M_n^{(E)}$. The argument leading to \eqref{estimating MnE} gives
    $$
    P^{n-\mathcal{D}}Q^{R}\sum_{J \subset \{1,\dots, n\}}(P^{-\frac{1}{3}}Q^3)^{n-| J| } \bigg\{\sum_{k \le (\log P)^B} \, \sideset{}{^\ast}\sum_{\chi \bmod k}\sum_{|\gamma| \le T}P^{\beta-1}\bigg\}^{| J|},
    $$
    where $J$ runs through all proper subsets of $\{1,\ldots,n \}$.
    It is plain that the right hand side of above is less than that in \eqref{estimating MnE}. Therefore, \eqref{MnE bound} yields
    \begin{equation}\label{MnE ' bound}
        M_n^{(E)}\ll P^{n-\mathcal{D}-\delta}
    \end{equation}
    for the present $M_n^{(E)}$, which is more than enough.

Inserting  \eqref{MnM ' bound} and \eqref{MnE ' bound} into \eqref{separate Mn '} proves the lemma.   
\end{proof}

\subsection{Estimation of $M_1,\ldots,M_{n-1}$}  
At this moment, it remains to show how to modify the previous argument to bound $M_1,\ldots,M_{n-1}$.
\begin{lemma}\label{lemma bound Mj}
    Let $M_j$ be as in \eqref{separate M into Mj}.
    Assume the conditions of Lemma~\ref{lem/maj/con}.
    Then, for $j=1,\ldots,n-1$,
    \begin{equation}\label{Mj bound}
        M_j \ll P^{n-\mathcal{D}}(\log P)^{-A},
    \end{equation}
    where $A>0$ is any fixed constant.
\end{lemma}
\begin{proof}
    Permuting the coordinates if necessary, we may assume that the contribution with minus signs in the first $j$ positions is maximal in absolute value. Thus
    \begin{equation}\label{using lemma bgs to bound Mj}
        \begin{split}
            M_j
       & \ll   \sum_{k_1 \le Q}\cdots \sum_{k_j \le Q} \, \sideset{}{^\ast}\sum_{\chi_1 \bmod k_1}\cdots \sideset{}{^\ast}\sum_{\chi_j \bmod k_j}   \\
       & \ \ \ \times \sum_{\substack{q \le Q \\ k_0|q}}\frac{1}{\varphi^n(q)} \bigg| \ \sideset{}{^\dagger }\sum_{\boldsymbol{a} \bmod q} \, \sideset{}{^\ast}\sum_{\boldsymbol{h} \bmod q}\bar{\chi}_1\chi^0(h_1)\cdots \bar{\chi}_j\chi^0(h_j)e\bigg(\frac{\boldsymbol{a}\cdot\boldsymbol{F}(\boldsymbol{h})}{q}\bigg)\bigg|\\
       & \ \ \ \times \bigg| \int  \sum_{ \boldsymbol{x}\in P\mathfrak{B}}
       \delta_1^- \cdots \delta_j^-e(\boldsymbol{\theta} \cdot\boldsymbol{F}(\boldsymbol{x})) d\boldsymbol{\theta}\bigg|, 
       \end{split}
    \end{equation}
    where $\delta^-_s=\delta^-(x_s,\chi_s\chi^0;q)$ is as in \eqref{define delta-} for $s=1,\ldots,j$, $\chi^0$ is the principal character modulo $q$ and
    $$
    k_0=[k_1,\ldots,k_j,1,\ldots,1]=[k_1,\ldots,k_j].
    $$
    It is easy to see that our treatments for $M_n$ work for $M_j$ in \eqref{using lemma bgs to bound Mj}, which gives \eqref{Mj bound}. Details are therefore omitted.
\end{proof}

\subsection{Proofs of Lemma \ref{lem/maj/con} and of Theorem~\ref{Thm2}}  

\begin{proof}[Proof of Lemma~\ref{lem/maj/con}] 
Inserting Lemmas \ref{lem/comp/M0}-\ref{lemma bound Mj} into \eqref{separate the integral on the major arcs} proves 
Lemma \ref{lem/maj/con}. 
\end{proof} 

\begin{proof}[Proof of Theorem~\ref{Thm2}]  
Theorem~\ref{Thm2} follows from Lemmas \ref{lemma over minor2} and \ref{lem/maj/con}. 
\end{proof} 

\medskip

\noindent{\bf Acknowledgements.} 
The authors would like to thank the anonymous referee
for a careful reading of the manuscript and for helpful
comments and suggestions.
The main results of this paper were obtained while S. Xie was a Ph.D. student at Shandong University.
The authors are grateful to Yang Cao, Bingrong Huang and Lilu Zhao for 
fruitful discussions, and to Xianchang Meng and Shuai Zhai for suggestions. 
The authors are supported by the National Key Research and Development 
Program of China (No. 2021YFA1000700), and the National Natural Science Foundation 
of China (Nos. 12031008 and 12631001).


\begin{thebibliography}{2010}

\bibitem{Bir57}
B.~J. Birch, Homogeneous forms of odd degree in a large number of variables,
Mathematika {\bf 4}, 102--105 (1957).

\bibitem{Bir61}
B.~J. Birch, Forms in many variables,
Proc. Roy. Soc. Ser. A {\bf 265}, 245--263 (1961/62).

\bibitem{BouGamSar}
J. Bourgain, A. Gamburd, and P. Sarnak,
Affine linear sieve, expanders, and sum-product,
Invent. Math. {\bf 179}, 559--644 (2010).

\bibitem{BraDie21}
J. Brandes and R. Dietmann, Rational lines on cubic hypersurfaces,
Math. Proc. Cambridge Philos. Soc. {\bf 171}, 99--112 (2021).

\bibitem{BroHB}
T.~D. Browning and D.~R. Heath-Brown, Forms in many variables and differing degrees,
J. Eur. Math. Soc. {\bf 19}, 357--394 (2017).

\bibitem{BDLW}
J. Br\"udern, R. Dietmann, J.~Y. Liu, and T.~D. Wooley,
A Birch-Goldbach theorem, Arch. Math. {\bf 94}, 53--58 (2010).

\bibitem{CooMag}
B. Cook and \'{A}. Magyar, Diophantine equations in the primes,
Invent. Math. {\bf 198}, 701--737 (2014).

\bibitem{Gre21}
B. Green, Quadratic forms in 8 prime variables,
Geom. Funct. Anal. {\bf 35}, 1587--1637 (2025).

\bibitem{GreTao08}
B. Green and T. Tao, The primes contain arbitrarily long arithmetic progressions,
Ann. of Math. {\bf 167}, 481--547 (2008).

\bibitem{Hua}
L.~K. Hua, Additive theory of prime numbers,
Transl. Math. Monogr., vol.~13, American Mathematical Society,
Providence, RI, 1965, xiii+190 pp.

\bibitem{IwaKow}
H. Iwaniec and E. Kowalski, Analytic number theory,
American Mathematical Society Colloquium Publications, vol.~53,
American Mathematical Society, Providence, RI, 2004.

\bibitem{KawWoo}
K. Kawada and T.~D. Wooley, On the Waring-Goldbach problem for fourth and fifth powers,
Proc. London Math. Soc. {\bf 83}, 1--50 (2001).

\bibitem{KumWoo}
A.~V. Kumchev and T.~D. Wooley, On the Waring-Goldbach problem for seventh and higher powers,
Monatsh. Math. {\bf 183}, 303--310 (2017).

\bibitem{Liu03}
J. Liu, On Lagrange's theorem with prime variables,
Quart. J. Math. {\bf 54}, 453--462 (2003).

\bibitem{Liu11}
J. Liu, Integral points on quadrics with prime coordinates,
Monatsh. Math. {\bf 164}, 439--465 (2011).

\bibitem{Liu12}
J. Liu, Enlarged major arcs in additive problems. II,
Tr. Mat. Inst. Steklova {\bf 276}, 182--197 (2012),
also in Proc. Steklov Inst. Math. {\bf 276}, 176--192 (2012).

\bibitem{LiuSar}
J. Liu and P. Sarnak, Integral points on quadrics in three variables whose
coordinates have few prime factors, Israel J. Math. {\bf 178}, 393--426 (2010).

\bibitem{LiuZha98}
J. Liu and T. Zhan, Sums of five almost equal prime squares II,
Sci. China Ser. A {\bf 41}, 710--722 (1998).

\bibitem{LiuZha23}
J. Liu and L. Zhao, On forms in prime variables,
Trans. Amer. Math. Soc. {\bf 376}, 8621--8656 (2023).

\bibitem{May15}
J. Maynard, Small gaps between primes,
Ann. of Math. {\bf 181}, 383--413 (2015).

\bibitem{Mon}
H.~L. Montgomery, Zeros of $L$-functions,
Invent. Math. {\bf 8}, 346--354 (1969).

\bibitem{Prachar}
K. Prachar, Primzahlverteilung, Springer, 1957.

\bibitem{Vin37}
I.~M. Vinogradov, Representation of an odd number as a sum of three primes,
Dokl. Akad. Nauk SSSR {\bf 15}, 291--294 (1937).

\bibitem{Yam18}
S. Yamagishi, Prime solutions to polynomial equations in many variables and
differing degrees, Forum Math. Sigma {\bf 6}, e19 (2018), 89 pp.

\bibitem{Yam22}
S. Yamagishi, Diophantine equations in primes: Density of prime points on
affine hypersurfaces, Duke Math. J. {\bf 171}, 831--884 (2022).

\bibitem{Zha14}
Y. Zhang, Bounded gaps between primes,
Ann. of Math. {\bf 179}, 1121--1174 (2014).

\bibitem{Zhao14}
L. Zhao, On the Waring-Goldbach problem for fourth and sixth powers,
Proc. London Math. Soc. {\bf 108}, 1593--1622 (2014).

\bibitem{Zhao16}
L. Zhao, The quadratic form in nine prime variables,
Nagoya Math. J. {\bf 223}, 21--65 (2016).
\end{thebibliography}
\end{document}